\documentclass[11pt]{article}

\usepackage[latin1]{inputenc}
\usepackage[english]{babel}
\usepackage[T1]{fontenc}
\usepackage{geometry}
\usepackage{amsmath, amsthm, amsfonts,bm}
\usepackage{amssymb}
\usepackage{amscd}
\usepackage{graphicx}
\usepackage{epsfig}
\usepackage{mathrsfs}
\usepackage{hyperref}
\usepackage{caption}
\usepackage{tikz}
\usepackage{array}
\usepackage{float}
\usepackage{url}
\usepackage{eepic}
\usepackage[all]{xy}
\usepackage{xcolor}
\usepackage{epstopdf}
\usepackage{mathtools}
\usepackage{enumerate}
\usepackage[toc]{appendix}

\newtheorem{thm}{Theorem}[section]

\newtheorem{lem}[thm]{Lemma}
\newtheorem{prop}[thm]{Proposition}
\theoremstyle{definition}
\newtheorem{defn}[thm]{Definition}
\theoremstyle{remark}

\newcommand{\cO}{{\mathcal{O}}}
\newcommand{\Om}{\Omega}
\newcommand{\wed}{\wedge}
\newcommand{\bN}{\mathbb{N}}
\newcommand{\proj}{\mathrm{proj}}
\newcommand{\del}{\partial}
\newcommand{\adel}{\overline{\partial}}
\newcommand{\rU}{\mathrm{U}}
\newcommand{\ltr}{\triangleleft}
\newcommand{\ltl}{\triangleright}
\newcommand{\btr}{\raisebox{0.2ex}{\hspace{-2pt} \scriptsize{\mbox{$\blacktriangleright$}}}}

\title{\textbf{Spectrum of the $\overline{\partial}$-Laplace operator on zero forms for the quantum quadric $\cO_q(\textbf{Q}_N)$}}
\author{Fredy D\'iaz Garc\'ia \vspace{3pt}\\
\small Mathematical Institute of National Autonomous University of Mexico \\
\small E-mail address: \emph{fredy@im.unam.mx}}

\date{}  

\begin{document}
\maketitle

\abstract{We study the Laplacian operator $\Delta_{\adel}$ associated to a K\"ahler structure $(\Omega^{(\bullet, \bullet)}, \kappa)$ for the Heckenberger--Kolb differential calculus of the quantum quadrics $\cO_q(\textbf{Q}_N)$, which is to say, the irreducible quantum flag manifolds of types $B_n$ and $D_n$. We show that the eigenvalues of $\Delta_{\adel}$ on zero forms tend to infinity and have finite multiplicity.}

\section{Introduction}

Since the discovery of quantum groups in the 1980s mathematicians have been trying to fit them into Connes' framework of noncommutative spin geometry. In fact, the construction of spectral triples for the Drinfeld--Jimbo quantum groups has been a very difficult problem in noncommutative geometry. One of the best-studied spectral triples is that of the Podle\'s sphere \cite{DS} which satisfies most of the Connes conditions for a noncommutative spin geometry, up to some slight modifications. Other constructions have been made in this direction, for example the equivariant isospectral Dirac operators for all the Podle\'s spheres \cite{DLE}, for the quantum projective spaces \cite{FD}, for the quantum group $\cO_q(SU(2))$ \cite{DLS} and for all the compact quantum groups \cite{NT}. In \cite{RO1} a Dirac operator associated to K\"ahler structures is constructed for the quantum projective spaces $\cO_q(\mathbb{CP}^n)$ generalizing the work of S. Majid in \cite{SM}.

In recent years it has been shown that the quantum flag manifolds provide a large family of well behaved quantum homogeneous spaces whose noncommutative geometry is quite similar to the classical situation. It was shown by I. Heckenberger and S. Kolb that the subclass of irreducible quantum flag manifolds $\cO_q(G/L_S)$ admit a unique differential calculus which is a $q$-deformation of their classical de Rham complex 
\cite{HK0}, \cite{HK}. 
It is important to note that the existence of this differential calculus is an extremely special property of these quantum spaces and it is not the case for all quantum spaces. Moreover the existence of this canonical deformation allows us to construct Dirac operators and investigate their spectrum.

The problem of the spectrum of the Dirac operator has been attacked for some cases. For example in \cite{FDL} the spectrum for $\cO_q(\mathbb{CP}^2)$ was calculated. In \cite{FD} the authors showed that the spectrum of the Dolbeault--Dirac operator for the quantum projective spaces has exponential growth by relating it to a Casimir operator. In \cite{KTS} the authors gave a new approach for  the construction of Dolbeault--Dirac operator and quantum Clifford algebras for the irreducible quantum flag manifolds. Later, it was shown in \cite{MM1} that this Dolbeault--Dirac operator gives a spectral triple for the quantum Lagrangian Grassmannian of rank 2 by giving a suitable version of the Parthasarathy formula.

Recently, a different approach to constructing Dolbeault--Dirac operators has appeared in \cite{RO} where  the notion of noncommutative K\"ahler structure for a differential calculus was introduced. In the same paper the author showed the existence of a covariant K\"ahler structure for the Heckenberger--Kolb calculus for the quantum projective spaces. Later this result was generalized in \cite{MM} for all the irreducible quantum flag manifolds, for all but a finite number of values $q$. The existence of such a structure allows us to construct metrics, Hilbert space completions and to verify all the axioms of a spectral triple except the compact resolvent condition, see \cite{RO2} for a more detailed discussion. This last axiom has turned out to be the most challenging one. This problem was solved in \cite{RO1} for the quantum projective spaces $\cO_q(\mathbb{CP}^n)$ by using the existence of the K\"ahler structure and the multiplicity-free property of the space of anti-holomorphic forms. Also, the spectrum was calculated for the quantum quadric $\cO_q(\textbf{Q}_5)$ in \cite{FRW} by using the quantum version of the Bernstein--Gelfand--Gelfand resolution for quantized irreducible flag manifolds \cite{HK1}. 

In this paper we investigate the asymptotic behaviour of the Laplace operator $\Delta_{\adel}$ on the space of zero forms $\Omega^{(0,0)}$ for the irreducible quantum flag manifolds of quadric type. Besides of the existence of the K\"ahler form we make use of the FRT description of the quantum group $\cO_q(SO(N))$ as well as the fact that  $\Omega^{(0,0)}$ is multiplicity-free as a $U_q(\mathfrak{g})$-module to make explicit calculations of the eigenvalues. It has been conjectured in \cite{RO1} that the Dolbeault--Dirac operator $D_{\adel}$ allows us to construct a spectral triple for all the irreducible quantum flag manifolds, here we take a small but significant step towards proving this conjecture.

The paper is organized as follows: in \textsection 2 we recall some necessary definitions and properties of the quantized enveloping algebras, differential calculi, K\"ahler structures as well as irreducible quantum flag manifolds and Takeuchi's equivalence. We also recall the FRT description of the algebras $\cO_q(SO(N))$.

In \textsection 3 we give the most important properties of the Heckenberger--Kolb calculus for the irreducible quantum flag manifolds and a proof of the existence of a K\"ahler structure for the case of quantum quadrics $\cO_q(\textbf{Q}_N)$ by using tools of representation theory.

In \textsection 4 we present the main result of the paper. It is shown by explicit calculation that the eigenvalues of the Laplace operator $\Delta_{\adel}$ on zero forms tend to infinity, supporting the conjecture that the asymptotic behaviour of the Dolbeault--Dirac operator $D_{\adel}$ satisfies the compact resolvent condition for all the irreducible quantum flag manifolds. 
 

\section{Preliminaries}

In this section we recall the necessary definitions of quantum enveloping algebras, covariant differential calculi as well as complex and K\"ahler structures. For a more detailed presentation see \cite{KS}, \cite{RO}, and \cite{ROC}. We also introduce the definition of quantum flag manifolds, Takeuchi's equivalence and the FRT presentation of the quantum coordinate algebra $\cO_q(SO(N))$.

\subsection{Drinfeld--Jimbo quantum groups}

Let $\mathfrak{g}$ a finite dimensional semisimple Lie algebra of rank $r$. Let $\mathfrak{g}$ be a finite-dimensional complex semisimple Lie algebra of rank $r$. We fix a Cartan subalgebra $\mathfrak{h}$ and choose a set of simple roots $\Pi = \{\alpha_1, \dots, \alpha_r\}$ for the corresponding root system in the linear dual of $\mathfrak{g}$. We denote by $(\cdot,\cdot)$ the symmetric bilinear form induced on $\mathfrak{h}^*$ by the Killing form of $\mathfrak{g}$, normalised so that any shortest simple root $\alpha_i$ satisfies $(\alpha_i,\alpha_i) = 2$. The Cartan matrix $(a_{ij})$ of $\mathfrak{g}$ is defined by $a_{ij} := \big(\alpha_i^{\vee},\alpha_j\big)$, where $\alpha_i^{\vee} := 2\alpha_i/(\alpha_i,\alpha_i)$ is the \emph{coroot} of $\alpha_i$.

Throughout the paper $q \in (0,+\infty)\setminus \{1\}$ and denote $q_i := q^{(\alpha_i, \alpha_i)/2}$ and the \emph{$q$-numbers} $[n]_q := \frac{q^n-q^{-n}}{q-q^{-1}}$. The {\em Drinfeld--Jimbo quantised enveloping algebra} $U_q(\mathfrak{g})$ is the noncommutative associative algebra generated by the elements $E_i, F_i, K_i$, and $K^{-1}_i$, for $ i=1, \ldots, r$, with the relations 
\begin{align*}
 K_iE_j = q_i^{a_{ij}} E_j K_i, ~~~~ K_iF_j= q_i^{-a_{ij}} F_j K_i, ~~~~ K_i K_j = K_j K_i, ~~~~ K_iK_i^{-1} = K_i^{-1}K_i = 1, \\
E_iF_j - F_jE_i = \delta_{ij}\frac{K_i - K^{-1}_{i}}{q_i-q_i^{-1}}, ~~~~~~~~~~~~~~~~~~~~~~~~~~~~~~~~~~~~~~
\end{align*}
and the quantum Serre relations which we omit (see \cite[\textsection 6.1.2]{KS} for details).

A Hopf algebra structure on $U_q(\mathfrak{g})$ is defined by 
\begin{align*}
\Delta(K_i) = K_i \otimes K_i, \ \ \Delta(E_i) = E_i \otimes K_i + 1\otimes E_i, \ \ \Delta(F_i) = F_i \otimes 1 + K_i^{-1} \otimes F_i, ~~~~ \\
S(K_i)= K^{-1}_i, \ \ S(E_i) = -E_i K^{-1}_i, \ \ S(F_i) = -K_i F_i, \ \ \varepsilon(K_i)=1, \ \ \varepsilon(E_i) = \varepsilon(F_i)=0.
\end{align*}
A Hopf $\ast$-algebra structure is given by
\begin{align*}
K_i^* = K_i, ~~~~~  E^*_i = K_i F_i, ~~~~~  F^*_i = E_i K_i^{-1}. 
\end{align*}
Let $\{\varpi_1, \ldots, \varpi_r\}$ denote the corresponding set of \emph{fundamental weights} of $\mathfrak{g}$ which is the dual basis of the simple coroots $\{\alpha_1^{\vee}, \ldots, \alpha_r^{\vee} \}$:
\begin{align*}
(\alpha_i^{\vee}, \varpi_j ) = \delta_{ij}.
\end{align*}
Let $\mathcal{P}$ be the weight lattice of $\mathfrak{g}$ and $\mathcal{P}^{+}$ the set of \emph{dominant integral weights} which are the $\mathbb{Z}$-span and $\mathbb{Z}_{\geq 0}$-span of the fundamental weights, respectively. For each $\mu \in  \mathcal{P}^+$ there exists an irreducible finite dimensional $U_q(\mathfrak{g})$-module $V_{\mu}$ uniquely determined by the existence of a vector $\nu_{\mu}$ which is called the \emph{highest weight vector}, satisfying
\begin{align*}
K_i \ltl \nu_{\mu} = q^{(\alpha_i, \mu)}\nu_{\mu}, ~~~~~  E_i \ltl \nu_{\mu}=0.
\end{align*}
Moreover $\nu_{\mu}$ is unique up to scalar multiple. A finite direct sum of such representations is called a \emph{type-1 representation}. A vector $\nu$ is called a \emph{weight vector} of weight $\mathrm{wt}(\nu)\in \mathcal{P}$ if 
\begin{align*}
K_i \ltl \nu = q^{(\alpha_i, \mathrm{wt}(\nu))}\nu.
\end{align*}  
In this paper we use the fact that $U_q(\mathfrak{g})$ has invertible antipode and therefore we have an equivalence between $\mbox{}_{U_q(\mathfrak{g})}$Mod and $\mathrm{Mod}_{U_q(\mathfrak{g})}$ induced by the antipode.



\subsection{Differential calculi}
A {\em differential calculus}  $\big(\Om^{\bullet} \simeq \bigoplus_{k \in \mathbb{N}_0} \Om^k, \mbox{d} \big)$ is a differential graded algebra (dg-algebra) which is generated in degree $0$ as a dg-algebra, that is to say, it is generated as an algebra by the elements $a, \mbox{d} b$, for $a,b \in \Om^0$. We call an element $\omega \in \Om^{\bullet}$ a \emph{form}, and if $\omega \in \Om^k$, for some $k \in \mathbb{N}_0$, then $\omega$ is said to be \emph{homogeneous} of degree $|\omega| := k$.  The product of two forms $\omega, \nu \in \Om^{\bullet}$ is usually denoted by $\omega \wedge \nu$, unless one of the forms is of degree $0$, where we just denote the product by juxtaposition.  

For a given algebra $B$, a \emph{differential calculus over} $B$ is a differential calculus such that $\Om^0 = B$. Note that for a differential calculus over $B$, each $\Om^k$ is a $B$-bimodule.  A differential calculus is said to have  \emph{total degree} $m \in \mathbb{N}_0$, if $\Om^m \neq 0$, and $\Om^k = 0$, for every $k > m$. \\

A {\em differential $*$-calculus} over a $*$-algebra $B$ is a differential calculus over $B$ such that the \mbox{$*$-map} of $B$ extends to a (necessarily unique) conjugate-linear involutive map $*:\Om^\bullet \to \Om^\bullet$ satisfying $\mbox{d}(\omega^*) = (\mbox{d} \omega)^*$, and 
\begin{align*}
\big(\omega \wed \nu \big)^*  =  (-1)^{kl} \nu^* \wed \omega^*, ~~~~~~ \text{ for all } \ \ \omega \in \Om^k, \, \nu \in \Om^l. 
\end{align*}

We say that $\omega \in \Om^{\bullet}$ is \emph{closed} if $\mbox{d} \omega = 0$, and \emph{real} if $\omega^*= \omega$.\\


We now recall the definition of a complex structure as introduced in \cite{ROC}. 
This abstracts the properties of the de Rham complex for a classical complex manifold \cite{DH}. 

\begin{defn}\label{defnnccs}
An {\em almost complex structure} $\Om^{(\bullet,\bullet)}$ for a  differential $*$-calculus  $(\Om^{\bullet},\mbox{d})$ is an $\mathbb{N}^2_0$-algebra grading $\bigoplus_{(a,b)\in \bN^2_0} \Om^{(a,b)}$ for $\Om^{\bullet}$ such that, for all $(a,b) \in \bN^2_0$ the following hold:
\begin{enumerate}[(i)]
\item \label{compt-grading}  $\Om^k = \bigoplus_{a+b = k} \Om^{(a,b)}$,
\item  \label{star-cond} $\big( \Om^{(a,b)} \big)^* = \Om^{(b,a)}$.
\end{enumerate}

We call an element of $\Om^{(a,b)}$ an $(a,b)$-form. For the associated pair of projections $\proj_{\Om^{(a+1,b)}} : \Omega^{a+b+1} \to \Omega^{(a+1,b)}$ and $\proj_{\Om^{(a,b+1)}} : \Omega^{a+b+1} \to \Omega^{(a,b+1)}$, we denote
\begin{align*}
\del|_{\Om^{(a,b)}} : = \proj_{\Om^{(a+1,b)}} \circ \mbox{d}, & & \adel|_{\Om^{(a,b)}} : = \proj_{\Om^{(a,b+1)}} \circ \mbox{d}.
\end{align*}
A  \emph{complex structure} is an almost complex structure which satisfies
\begin{align} \label{eqn:integrable}
\mbox{d} \Om^{(a,b)} \subseteq \Om^{(a+1,b)} \oplus \Om^{(a,b+1)}, ~~~~~~ \text{ for all } \ (a,b) \in \bN^2_0.
\end{align}
\end{defn}

For a complex structure, (\ref{eqn:integrable}) implies the identities
\begin{align*}
\mbox{d} = \del + \adel, & &  \adel \circ \del = - \, \del \circ \adel, & & \del^2 = \adel^2 = 0. 
\end{align*}
In particular, $\big(\bigoplus_{(a,b)\in \bN^2_0}\Om^{(a,b)}, \del, \adel\big)$ is a double complex, which we call  the {\em Dolbeault double complex} of $\Om^{(\bullet,\bullet)}$. Moreover, it is easily seen  that both $\del$ and $\adel$ satisfy the graded Leibniz rule, and  
\begin{align}\label{dif.inv.}
\del(\omega^*) = \big(\adel \omega \big)^*, ~~~~~~~~~ \adel(\omega^*) = \big(\del \omega\big)^*, ~~~~~~~~ \text{for all} ~~~~ \omega \in \Om^{\bullet}.
\end{align}

\subsubsection{Hermitian and K\"ahler structures}

We now present the definition of an Hermitian structure, as introduced in \cite{RO}, as well as a K\"ahler structure, introduced in \textsection7 of the same paper.

\begin{defn} An {\em Hermitian structure} $(\Om^{(\bullet,\bullet)}, \sigma)$ for a differential $*$-calculus $\Om^{\bullet}$ over $B$ of even total degree $2n$  is a pair  consisting of  a complex structure  $\Om^{(\bullet,\bullet)}$ and  a central real $(1,1)$-form $\sigma$, called the {\em Hermitian form}, such that, with respect to the {\em Lefschetz operator}
\begin{align*}
L:\Om^{\bullet} \to \Om^{\bullet},  & &   \omega \mapsto \sigma \wedge \omega,
\end{align*}
isomorphisms are given by
\begin{align*}
L^{n-k}: \Om^{k} \to  \Om^{2n-k}, & & \text{ for all } \ \ k = 0, \dots, n-1.
\end{align*}
\end{defn} 

For $L$ the Lefschetz operator of an Hermitian structure, we denote
\begin{align*}
P^{(a,b)} : = \begin{cases} 
 \{\alpha \in \Om^{(a,b)} \,|\, L^{n-a-b+1}(\alpha) = 0\}, &  \text{ ~ if } a+b \leq n,\\
      0, & \text{ ~ if } a+b > n.
\end{cases}
\end{align*}
Moreover, we denote $P^k := \bigoplus_{a+b = k} P^{(a,b)}$, and $P^{\bullet} := \bigoplus_{k \in \bN_0} P^k$. An element of $P^{\bullet}$ is called a {\em primitive form}. 

An important consequence of the existence of the Lefschetz operator is the Lefschetz decomposition of the differential $*$-calculus, which we now recall below. For a proof see \cite[\textsection 4.1]{RO}.

\begin{prop}[Lefschetz decomposition] \label{LDecomp}
For $L$ the Lefschetz operator of an Hermitian structure on a differential $*$-calculus $\Omega^{\bullet}$, a $B$-bimodule decomposition of $\Omega^{k}$, for all $k \in \bN_0$,  is given by
\begin{align*}
\Om^{k} \simeq \bigoplus_{j \geq 0} L^j \big( P^{k-2j} \big).
\end{align*}
\end{prop}

In classical Hermitian geometry, the Hodge map of an Hermitian metric is related to the associated Lefschetz decomposition through the well-known Weil formula \cite[\textsection 1.2]{DH}. In the noncommutative setting, we take the direct generalisation of the Weil formula for our definition of the Hodge map, and build upon this to define an Hermitian metric.

\begin{defn} \label{defn:HDefn}
The {\em Hodge map} associated to an Hermitian structure $\big(\Om^{(\bullet,\bullet)},\sigma \big)$ is the $B$-bimodule map $\ast_{\sigma}: \Omega^{\bullet} \to \Omega^{\bullet}$ satisfying, for any $j \in \bN_0$,
\begin{align*}
\ast_{\sigma}\big( L^j(\omega) \big) = (-1)^{\frac{k(k+1)}{2}}\mathbf{i}^{a-b}\frac{j!}{(n-j-k)!}L^{n-j-k}(\omega), & & \omega \in P^{(a,b)} \subseteq P^{k=a+b},
\end{align*}
where $\mathbf{i}^2 = -1$.
\end{defn}

The metric associated to a Hermitian structure $(\Om^{(\bullet, \bullet)}, \sigma)$ is the map $g_{\sigma} : \Omega^{\bullet} \times \Omega^{\bullet} \rightarrow B$, for which $g(\Om^k, \Om^l)=0$ for $k\neq l$ and 
\[ g_{\sigma} (\omega, \nu) :=  *_{\sigma} ( *_{\sigma}(\omega^*) \wedge \nu ), ~~~~~ \text{ if }  \ \omega, \nu \in \Om^k. \]



\begin{defn} \label{defn:posdefHerm}
We say that an Hermitian structure $(\Omega^{(\bullet,\bullet)}, \sigma)$ is \emph{positive definite} if the associated metric $g_{\sigma}$ satisfies 
\begin{align*}
g_{\sigma}(\omega, \omega) \in B_{>0}:= \Big\{ \sum_{i=1}^{m} \lambda_i b^{*}b_i : b_i \in B, \ \ \lambda_i \in \mathbb{R}_{>0}, \ m\in \mathbb{N} \Big\}, \ \ \mbox{for all non-zero} \ \ \omega \in \Omega^{\bullet}.
\end{align*}
In this case we say that $\sigma$ is a {\em positive definite Hermitian form}. 
\end{defn}

\begin{defn}
A {\em K\"ahler structure} for a differential $*$-calculus is an Hermitian structure $(\Om^{(\bullet,\bullet)},\kappa)$ such that $\kappa$ is closed, which is to say, $\mbox{d} \kappa = 0$. We call such a form $\kappa$ a  {\em K\"ahler form}.
\end{defn}

\subsection{Irreducible quantum flag manifolds}\label{quantumflags}
Let $V$ be a finite-dimensional $U_q(\mathfrak{g})$-module, $v \in V$, and $f \in V^*$, the linear dual of $V$. Consider the function $c^{\textrm{\tiny $V$}}_{f,v}: U_q(\mathfrak{g}) \to \mathbb{C}$ defined by $c^{\textrm{\tiny $V$}}_{f,v}(X) := f\big(X(v)\big)$.
The {\em coordinate ring} of $V$ is the subspace
\begin{align*}
C(V) := \text{span}_{\mathbb{C}}\!\left\{ c^{\textrm{\tiny $V$}}_{f,v} \,| \, v \in V, \, f \in V^*\right\} \subseteq U_q(\mathfrak{g})^*.
\end{align*}
It is easy to check that $C(V)$ is contained in $U_q(\mathfrak{g})^\circ$, the Hopf dual of $U_q(\mathfrak{g})$, and moreover that a Hopf subalgebra of $U_q(\mathfrak{g})^\circ$ is given by 
\begin{align} \label{eqn:PeterWeyl}
\mathcal{O}_q(G) := \bigoplus_{\lambda \in \mathcal{P}^+} C(V_{\lambda}).
\end{align}
This algebra can be endowed with a dual $\ast$-structure and it is called the {\em quantum coordinate algebra of $G$}, where $G$ is the compact, simply-connected, simple Lie group having $\mathfrak{g}$ as its complexified Lie algebra. Moreover, we call the decomposition of $\cO_q(G)$ given in \eqref{eqn:PeterWeyl} the \emph{Peter--Weyl decomposition} of $\cO_q(G)$.

For $S$ a subset of simple roots, consider the Hopf \mbox{subalgebra} of $U_q(\mathfrak{g})$ given by 
\begin{align*}
U_q(\mathfrak{l}_S) := \big< K_i, E_j, F_j \,|\, i = 1, \ldots, r ; j \in S \big>.
\end{align*} 
There are obvious notions of weight vectors, highest weight vectors as well as type-1 representations. It can be shown that every irreducible finite dimensional type-1 $U_q(\mathfrak{l}_S)$-module admits a highest weight vector which is unique up to scalar multiple. Moreover, the weight of this highest weight vector determines the module up to isomorphism and the there is a correspondence between irreducible type-1 $U_q(\mathfrak{l}_S)$ modules and weights in the sub-lattice
\begin{align*}
\mathcal{P}^+_S + \mathcal{P}_{S^c} \subseteq \mathcal{P},
\end{align*}
where $S^c = \Pi \setminus S$ and
\begin{align*}
\mathcal{P}^+_S &:= \Big\{ \lambda \in \mathcal{P}: \lambda = \sum_{j\in S} \lambda_j \varpi_j, \ \ \lambda_j \in  \mathbb{Z}_{\geq 0} \Big\},  \\
\mathcal{P}_{S^c} &:= \Big\{ \lambda \in \mathcal{P}: \lambda = \sum_{j\in S^c} \lambda_j \varpi_j \ \ \lambda_j \in \mathbb{Z} \Big\}.  
\end{align*}
Just as for $\cO_q(G)$, we can construct the type-1 dual of $U_q(\mathfrak{l}_S)$ using matrix coefficients, this Hopf algebra will be denoted by $\cO_q(L_S)$. Restriction of domains gives us a surjective Hopf $\ast$-algebra map $\pi_S: \cO_q(G) \rightarrow \cO(L_S)$, dual to the inclusion of $U_q(\mathfrak{l}_S)$ in $U_q(\mathfrak{g})$. The {\em quantum flag manifold associated to S} is the quantum homogeneous space associated to $\pi_S$
\begin{align*}
\cO_q \big(G/L_S \big) &:= \cO_q(G)^{co(\cO_q(L_S))} \\
& \ = \{ b\in \cO_q(G): b_{(1)}\langle X , b_{(2)} \rangle = \varepsilon(X)b  ~~~~ \text{for all} \ \  X\in U_q(\mathfrak{l}_S) \}.
\end{align*}  
We say that the quantum flag manifold is irreducible if $S= \Pi \setminus \{\alpha_x \}$ where $\alpha_x$ has coefficient $1$ in the expansion of the highest root of $\mathfrak{g}$. In the table \ref{table:CQFMs} we give a diagrammatic presentation of the set of simple roots defining the irreducible quantum flag manifold, where the node corresponding to $\alpha_x$ is the coloured one.

\begin{center}
\begin{table}[ht] 


\centering 

\captionof{table}{\small{Irreducible Quantum Flag Manifolds with nodes numbered according to \cite[\textsection11.4]{JH}.}}\label{table:CQFMs}

{\small \renewcommand{\arraystretch}{2}

\begin{tabular}{|c|c|c|c|}
 
\hline

\small $A_n$ &
\begin{tikzpicture}[scale=.4]
\draw
(0,0) circle [radius=.25] 
(8,0) circle [radius=.25] 
(2,0)  circle [radius=.25]  
(6,0) circle [radius=.25] ; 

\draw[fill=black]
(4,0) circle  [radius=.25] ;

\draw[thick,dotted]
(2.25,0) -- (3.75,0)
(4.25,0) -- (5.75,0);

\draw[thick]
(.25,0) -- (1.75,0)
(6.25,0) -- (7.75,0);
\end{tikzpicture} & \small $\cO_q(\text{Gr}_{s,n+1})$ & \small quantum Grassmannian  \\


\small $B_n$ &
\begin{tikzpicture}[scale=.4]
\draw
(4,0) circle [radius=.25] 
(2,0) circle [radius=.25] 
(6,0)  circle [radius=.25]  
(8,0) circle [radius=.25] ; 
\draw[fill=black]
(0,0) circle [radius=.25];

\draw[thick]
(.25,0) -- (1.75,0);

\draw[thick,dotted]
(2.25,0) -- (3.75,0)
(4.25,0) -- (5.75,0);

\draw[thick] 
(6.25,-.06) --++ (1.5,0)
(6.25,+.06) --++ (1.5,0);                      

\draw[thick]
(7,0.15) --++ (-60:.2)
(7,-.15) --++ (60:.2);
\end{tikzpicture} & \small $\cO_q(\mathbf{Q}_{2n+1})$ & \small {odd} quantum quadric  \\ 


\small $C_n$& 
\begin{tikzpicture}[scale=.4]
\draw
(0,0) circle [radius=.25] 
(2,0) circle [radius=.25] 
(4,0)  circle [radius=.25]  
(6,0) circle [radius=.25] ; 
\draw[fill=black]
(8,0) circle [radius=.25];

\draw[thick]
(.25,0) -- (1.75,0);

\draw[thick,dotted]
(2.25,0) -- (3.75,0)
(4.25,0) -- (5.75,0);

\draw[thick] 
(6.25,-.06) --++ (1.5,0)
(6.25,+.06) --++ (1.5,0);                      

\draw[thick]
(7,0) --++ (60:.2)
(7,0) --++ (-60:.2);
\end{tikzpicture} &\small   $\cO_q(\mathbf{L}_{n})$ & \small 
quantum Lagrangian Grassmannian    \\ 

\small $D_n$& 
\begin{tikzpicture}[scale=.4]

\draw[fill=black]
(0,0) circle [radius=.25] ;

\draw
(2,0) circle [radius=.25] 
(4,0)  circle [radius=.25]  
(6,.5) circle [radius=.25] 
(6,-.5) circle [radius=.25];

\draw[thick]
(.25,0) -- (1.75,0)
(4.25,0.1) -- (5.75,.5)
(4.25,-0.1) -- (5.75,-.5);

\draw[thick,dotted]
(2.25,0) -- (3.75,0);
\end{tikzpicture} &\small   $\cO_q(\mathbf{Q}_{2n})$ & \small even quantum quadric  \\

\small $D_n$ & 
\begin{tikzpicture}[scale=.4]
\draw
(0,0) circle [radius=.25] 
(2,0) circle [radius=.25] 
(4,0)  circle [radius=.25] ;

\draw[fill=black] 
(6,.5) circle [radius=.25];
\draw
(6,-.5) circle [radius=.25];

\draw[thick]
(.25,0) -- (1.75,0)
(4.25,0.1) -- (5.75,.5)
(4.25,-0.1) -- (5.75,-.5);

\draw[thick,dotted]
(2.25,0) -- (3.75,0);
\end{tikzpicture} &\small   $\cO_q(\textbf{S}_{n})$ & \small quantum spinor variety   \\

\small $E_6$& \begin{tikzpicture}[scale=.4]
\draw
(2,0) circle [radius=.25] 
(4,0) circle [radius=.25] 
(4,1) circle [radius=.25]
(6,0)  circle [radius=.25] ;

\draw
(0,0) circle [radius=.25];
\draw[fill=black] 
(8,0) circle [radius=.25];

\draw[thick]
(.25,0) -- (1.75,0)
(2.25,0) -- (3.75,0)
(4.25,0) -- (5.75,0)
(6.25,0) -- (7.75,0)
(4,.25) -- (4, .75);
\end{tikzpicture}

 &\small  $\cO_q(\mathbb{OP}^2)$ & \small  quantum Caley plane   \\
\small $E_7$& 
\begin{tikzpicture}[scale=.4]
\draw
(0,0) circle [radius=.25] 
(2,0) circle [radius=.25] 
(4,0) circle [radius=.25] 
(4,1) circle [radius=.25]
(6,0)  circle [radius=.25] 
(8,0) circle [radius=.25];

\draw[fill=black] 
(10,0) circle [radius=.25];

\draw[thick]
(.25,0) -- (1.75,0)
(2.25,0) -- (3.75,0)
(4.25,0) -- (5.75,0)
(6.25,0) -- (7.75,0)
(8.25, 0) -- (9.75,0)
(4,.25) -- (4, .75);
\end{tikzpicture} &\small   $\cO_q(\textrm{F})$ 
& \small   quantum Freudenthal variety   \\
\hline
\end{tabular}
}
\end{table}
\end{center}

For the irreducible $U_q(\mathfrak{g})$-module $V_{\varpi_x}$ we choose a weight basis $\{v_1, \ldots, v_N\} \subset V_{\varpi_x}$ with corresponding dual basis $\{f_1, \ldots , f_N\}\subset V_{\varpi_x}^{*}$, where $N= \mathrm{dim}(V_{\mu})$.  
It was shown in \cite[Theorem 4.1]{HK} that a set of generators for $\mathcal{O}_q(G/L_S)$ in the irreducible case is given by 
\begin{align}\label{generators}
z_{ij}:= c^{\varpi_x}_{f_i,v_N}c^{-w_0 \varpi_x}_{v_j,f_N} ~~~~~~~   \text{for} \ \ i,j=1, \ldots, N=\mathrm{dim}(V_{\varpi_x}),
\end{align}
where $v_N$ is the highest weight basis element of $V_{\varpi_x}$ and $f_N$ is the lowest weight basis element of $V_{-w_0(\varpi_x)}$. Note that if $v_1$ is the lowest weight then $f_1$ is the highest weight. 
\subsubsection{Takeuchi's equivalence}
In the following we denote $\mbox{}^{~~~~\cO_q(G)}_{\cO_q(G/L_S)}\mathrm{Mod}_0$ the category whose objects are relative Hopf modules $\mathcal{F}$ such that $\mathcal{F}\cO_q(G/L_S)^{+} = \cO_q(G/L_S)^{+} \mathcal{F}$ and morphisms are left $\cO_q(G)$-comodule, $\cO(G/L_S)$-bimodule maps where $\cO_q(G)^{+}=\cO_q(G)\cap \mathrm{Ker}(\varepsilon)$. Let $\mbox{}^{\cO(L_S)}\mathrm{Mod}$ be the category whose objects are $\cO(L_S)$-comodules and morphisms are $\cO_q(L_S)$-comodule maps. By \emph{Takeuchi's equivalence}, see for example \cite{MT}, \cite[\textsection 2.2]{HK}, an equivalence of the categories $\mbox{}^{~~~~\cO_q(G)}_{\cO_q(G/L_S)}\mathrm{Mod}_{0}$ and $\mbox{}^{\cO_q(L_S)}\mathrm{Mod}$ is given by the functors
\begin{align*}
& \Phi: \mbox{}^{~~~~\cO_q(G)}_{\cO_q(G/L_S)}\mathrm{Mod}_0 \rightarrow \mbox{}^{\cO_q(L_S)}\mathrm{Mod}, ~~~~~~~~ \Phi(\Gamma) = \Gamma/\cO_q(G/L_S)^{+}\Gamma,  \\
& \Psi: \mbox{}^{\cO_q(L_S)}\mathrm{Mod} \rightarrow \mbox{}^{~~~ \cO_q(G)}_{\cO(G/L_S)}\mathrm{Mod}_0, ~~~~~~~~~  \Psi(V)= \cO_q(G) \square_{\cO_q(L_S)} V.   
\end{align*} 
Here for $\Gamma \in \mbox{}^{~~~\cO_q(G)}_{\cO(G/L_S)}\mathrm{Mod}_0$ the left $\cO(L_S)$-comodule structure on $\Gamma/\cO_q(G/L_S)^{+}\Gamma$ is induced by the left $\cO_q(G)$-comodule structure of $\Gamma$. The left $\cO_q(G)$-comodule and $\cO(G/L_S)$-bimodule structures on the cotensor $\cO_q(G) \square_{\cO_q(L_S)} V$ are given by the coproduct and $\cO(G/L_S)$-bimodule structure of $\cO_q(G)$, respectively. Takeuchi's equivalence states that natural isomorphisms are given by 
\begin{align*}
& \mathrm{C}:  \Phi \circ \Psi(V) \rightarrow V,   ~~~~~~~  \Big[ \sum_i a^i \otimes v^i \Big] \mapsto \sum_i \varepsilon(a^i) v^i, \\
& \mathrm{U}: \mathcal{F} \rightarrow \Psi \circ \Phi(\mathcal{F}), ~~~~~~~~~~~~~~~  f \mapsto f_{(-1)}\otimes [f_{(0)}]. 
\end{align*}

\subsection{The Hopf algebras $\cO_q(SO(N))$}

In this subsection we give a brief description of the FRT presentation of the Hopf algebras $\cO_q(SO(N))$, see \cite{KS} for details. We first give a bit of notation. We denote $i'= N+1 -i$. Let $N=2n$ if $N$ is even and $N=2n+1$ if $N$ is odd. We define 
\begin{align*}
\rho_i = \frac{N}{2}-i \ \ \text{if} \ \  i<i', \ \  \rho_{i'} = -\rho_i \ \ \text{if} \ \  i\leq i'. 
\end{align*}

Recall that the R-matrix associated to the quantum $\cO_q(SO(N))$ is given by 

\begin{equation}
R^{ij}_{mn} = q^{\delta_{ij}-\delta_{ij'}}
\delta_{im}\delta_{jn}+(q-q^{-1})\theta(i-m)(\delta_{jm}\delta_{in}- \delta_{ji'}\delta_{mn'}q^{-\rho_{j}-\rho_{m}}).
\end{equation}

Let $U_q(\mathfrak{so}_N)^{\circ}$ denote the dual Hopf algebra of $U_q(\mathfrak{so}_N)$. Then $\cO_q(SO(N)) \subset U_q(\mathfrak{so}_N)^{\circ}$ is the Hopf subalgebra generated by matrix coefficients of the $N$-dimensional irreducible $U_q(\mathfrak{so}_N)$-representation of highest weight $\varpi_1$\footnote{$\cO_q(SO(2n+1))$ is a proper subalgebra of $\cO_q(G)$ because $\cO_q(SO(2n+1))$ is just generated by the matrix coefficients of $V_{\varpi_1}$ and it does not contain all the matrix coefficients of all the representations of type 1, for example,  it doest not contain the matrix coefficients of the spin representations, but $\cO_q(G)^{U_q(\mathfrak{l})} = \cO_q(SO(2n+1))^{U_q(\mathfrak{l})}$.}. We denote the matrix coefficients by $u^i_j$, $i,j=1, \ldots, N$, they satisfy the relations
\begin{align}\label{RmatrixRel}
& \sum _{k,l=1}^{N} R^{ij}_{kl}u^k_mu^l_n = \sum_{k,l=1}^N u^j_ku^i_lR^{lk}_{mn}, \ \ i,j,m,n = 1, \ldots, N, \\
& \mathcal{D}_q = 1,  \ \  \sum_{i,j,k=1}^N C^i_j(C^{-1})^k_mu^n_iu^k_j = \sum_{i,j,k=1}^N C^n_i(C^{-1})^j_ku^j_iu^k_m = \delta_{mn}, \ \ m,n=1, \ldots,N,
\end{align}
where the C-matrix coefficients and the quantum determinant are given as in \cite[\textsection 9.3]{KS}. The Hopf $\ast$-algebra structure on  $\cO_q(SO(N))$ is determined on generators by
\begin{align*}
\Delta(u^i_j) = \sum_k u^i_k \otimes u^k_j, ~~~~~ \varepsilon(u^i_j)= \delta_{ij}, ~~~~~ S(u^i_j)= q^{\rho_j - \rho_i}u^{j'}_{i'}, ~~~~~ u^{i*}_j= S(u^j_i).
\end{align*}

It follows from \cite[\textsection 9.4]{KS} that there exists a dual pairing $\langle \cdot , \cdot \rangle$ of the Hopf algebras $U_q(\mathfrak{so}_N)$ and $\cO_q(SO(N))$ which in turn induces a left action $U_q(\mathfrak{so}_N) \otimes \cO_q(SO(N)) \ni  X \otimes a \mapsto X \ltl a \in \cO_q(SO(N))$
and a right action $\cO_q(SO(N))\otimes U_q(\mathfrak{so}_N) \ni a \otimes X \mapsto a \ltr X \in \cO_q(SO(N))$ 
such that $\cO_q(SO(N))$ is a left and a right $U_q(\mathfrak{so}_N)$-module $\ast$-algebra. The actions on generators are given as follows: 
\begin{align}
& E_j \ltl u^i_j = u^i_{j+1}, ~~~  E_j \ltl u^i_{(j+1)'} = -u^i_{j'}, ~~~ F_{j} \ltl u^i_{j+1} = u^i_{j}, ~~~ F_j \ltl u^i_{j'}= -u^i_{(j+1)'},  \  \  j<n, \label{actEiodd} \\
& E_n \ltl u^i_n = [2]^{1/2}_{q_2} u^i_{n+1}, ~~~  E_n \ltl u^i_{n+1} = -q_2 [2]^{1/2}_{q_2}u^i_{n+2}, \\
& F_n \ltl u^i_{n+1} = [2]^{1/2}_{q_2} u^i_{n}, ~~~ F_n \ltl u^i_{n+2} = -q^{-1}_2[2]^{1/2}_{q_2} u^i_{n+1}, 
\end{align}
for $i=1, \ldots, N$ and $N=2n+1$ where $q_2= q^{1/2}$. For $N=2n$ we have 
\begin{align}
& E_j \ltl u^i_j =  u^i_{j+1}, ~~~  E_j \ltl u^i_{(j+1)'} = -u^i_{j'}, ~~~ F_j \ltl u^i_{j+1} = u^i_j, ~~~   F_j \ltl u^i_{j'} = u^i_{(j+1)'}  \  \   i<n, \\
& E_n \ltl u^i_n = -u^i_{n+2}, ~~~  E_n \ltl u^i_{n-1} = u^i_{n+1},  \\
& F_n \ltl u^i_{n+2} = -u^i_{n}, ~~~ F_n \ltl u^i_{n+1} = u^i_{n-1}. \label{actEieven}
\end{align}
The corresponding right actions for $N=2n+1$ are given by
\begin{align}
&  u^i_j \ltr F_i = u^{i+1}_{j}, ~~~  u^{(i+1)'}_{j} \ltr F_i = -u^{i'}_{j}, ~~~  u^{i+1}_{j} \ltr E_i = u^i_{j}, ~~~ u^{i'}_{j} \ltr E_i = -u^{(i+1)'}_{j},  \  \  i<n, \label{actEioddr} \\
& u^n_j \ltr F_n = [2]^{1/2}_{q_2} u^{n+1}_j, ~~~  u^{n+1}_j \ltr F_n = -q_2 [2]^{1/2}_{q_2}u^{n+2}_j, \\
& u^{n+1}_j \ltr E_n = [2]^{1/2}_{q_2} u^{n}_j, ~~~  u^{n+2}_j \ltr E_n = -q^{-1}_2[2]^{1/2}_{q_2} u^{n+1}_j, \label{actEnoddr} 
\end{align}
and for even $N=2n$
\begin{align}
& u^i_j \ltr F_i =  u^{i+1}_{j}, ~~~  u^{(i+1)'}_j \ltr F_i = -u^{i'}_j, ~~~ u^{i+1}_j \ltr E_i = u^i_j, ~~~ u^{i'}_j \ltr E_i = u^{(i+1)'}_j,  \  \   i<n, \label{actEievenr} \\
& u^n_j \ltr F_n = -u^{n+2}_j, ~~~  u^{n-1}_j \ltr F_n  = u^{n+1}_j,  \\
& u^{n+2}_j \ltr E_n = -u^{n}_j, ~~~ u^{n+1}_j \ltr E_n = u^{n-1}_j, \label{actEnevenr}
\end{align}
for $i=1, \ldots, N$. The values $E_j \ltl u^k_l$, $F_j \ltl u^k_l$, $u^k_l \ltr E_j$ and $u^k_l \ltr F_j$ for all other cases are zero.

For the convenience of the reader, the next lemma collects the commutation relations of the generators of $\cO_q(SO(N))$ which will be used in this paper.
\begin{lem}\label{Rels}
If $u^i_j, i,j =1, ..., N$ are the generators of $\mathcal{O}_q(SO(N))$ then the following commutation relations hold
\begin{align}
&u^i_1u^i_N = q^2 u^i_Nu^i_1, \ i\neq i', ~~~ \label{ui1uiN} \\
&u^i_lu^i_k = q u^i_ku^i_l, \ l<k, \ i\neq i', \\
&u^j_lu^i_k = u^i_ku^j_l, \ l<k, \ i<j, \ l\neq k', \ i\neq j', \label{ujluik} \\
&u^j_1u^i_N = q u^i_Nu^j_1, \ i <j, \ i\neq j',  \\
&u^i_1u^j_N = qu^j_Nu^i_1 + (q^2 -1)u^i_N u^j_1,  \ i <j, \ i\neq j', \label{uiNu1} \\
&u^i_ku^j_l = u^j_lu^i_k - (q-q^{-1})u^j_ku^i_l, \ i<j, \ k<l, \ i\neq j', \ k\neq l', \label{uikujl} \\
&u^i_k u^j_k = qu^j_k u^i_k, \ k\neq k', \ i<j, \ i \neq j', \label{uikujk} \\
&(u^1_1u^2_a - q u^2_1 u^1_a )(u^1_1 u^2_N - q u^2_1u^1_N) = q^2(u^1_1 u^2_N - qu^2_1u^1_N)(u^1_1u^2_a - q u^2_1 u^1_a ),  \label{hwhol} \\
&(u^1_1u^2_N - qu^2_1u^1_N)( u^1_{a} u^2_N - q u^2_{a} u^1_N ) = q^2 ( u^1_{a} u^2_N - q u^2_{a} u^1_N )(u^1_1u^2_N - qu^2_1u^1_N), \label{hwantihol}
\end{align}
where  $a=2, \ldots, N-1$.
\end{lem}

\begin{proof}
Equations \eqref{ui1uiN}-\eqref{uikujl} follow directly from the R-matrix relations \eqref{RmatrixRel}. For instance, if  we take $m=N$, $n=1$ and $i>j$ with $i\neq j'$ in relation \eqref{RmatrixRel} we obtain 
\begin{align*}
u^i_N u^j_1 +(q-q^{-1})u^j_N u^i_1 = q^{-1} u^j_1 u^i_N,  
\end{align*}
which yields equation \eqref{uiNu1}. Equation \eqref{uikujk} follows by choosing $k=l+1$ in \eqref{ujluik} and acting on both sides by the $E_i$'s or  $F_i$'s from the left by using formulas \eqref{actEiodd}-\eqref{actEieven}. We only give the proof of \eqref{hwantihol} since the proof of \eqref{hwhol} is analogous. By using \eqref{ui1uiN}-\eqref{uikujk} it can be show that $u^1_N(u^1_1 u^2_N -q u^2_1 u^1_N) = q^{-1} (u^1_1 u^2_N -q u^2_1 u^1_N)u^1_N$. If we act on both sides by $E_1$ from the left we get 
\begin{align*}
u^1_N(u^1_2 u^2_N -q u^2_2 u^1_N) = q^{-1} (u^1_2 u^2_N -q u^2_2 u^1_N)u^1_N. 
\end{align*} 
Then the relation $u^1_N(u^1_a u^2_N -q u^2_a u^1_N) = q^{-1} (u^1_a u^2_N -q u^2_au^1_N)u^1_N$ can be obtained by acting from the left by the $E_i$'s according to identities \eqref{actEiodd}-\eqref{actEieven}. The commutation relations 
\begin{align*}
u^1_1(u^1_a u^2_N - q u^2_a u^1_N) = q^2(u^1_a u^2_N - q u^2_a u^1_N)u^1_1, \quad \    u^2_N        (u^1_a u^2_N - q u^2_a u^1_N)= (u^1_a u^2_N - q u^2_a u^1_N)u^2_N,  \\
   u^2_1 ( u^1_a u^2_N - q u^2_a u^1_N ) = (u^1_a u^2_N - q u^2_a u^1_N) u^2_1 ~~~~~~~~~~~~~~~~~~~~~~~~~~~~~~~~~~
\end{align*}
can be easily deduced from \eqref{ui1uiN}-\eqref{uikujk}. Combining this results we have that \eqref{hwantihol} holds for $a=2, \ldots,N-1$. 
\end{proof}
For the $B_n$ and $D_n$ family in which case $\varpi_x =\varpi_1$ we can view $\cO_q(SO(N))\subset \cO_q(G)$ by choosing a weight basis $\{v_j\}_{j=1}^N$ of $V_{\varpi_1}$ such that $c^{\varpi_1}_{f_i,v_j}= u^i_j$ and $c^{-w_0 (\varpi_1)}_{v_j,f_i}=S(u^i_j)$, see \cite{FRT} or \cite{KS} for details. Therefore by \eqref{generators} the generators of $\cO_q(\mathrm{Q}_N)$ are given by 
\begin{align*}
z_{ij} = 
u^i_N S(u^N_j).
\end{align*} 

\section{The Heckenberger--Kolb calculi and the K\"ahler structure}


In this section we recall important facts about Heckenberger--Kolb calculi for irreducible quantum flag manifolds as well as the existence of a K\"ahler structure for them. For more details we refer the reader to the seminal papers \cite{HK} and \cite{HK1} for the Heckenberger--Kolb calculus and \cite{MM}, \cite{RO} for K\"ahler structures. 

As mentioned before, the irreducible quantum flag manifolds are distinguished by existence of a unique $q$-deformed de Rham complex. The following theorem collects the main results of \cite{HK0}, \cite{HK}, and \cite{MM}.  
\begin{thm} 
\label{HKcalc}
For any irreducible quantum flag manifold $\cO_q(G/L_S)$, there exists a unique finite dimensional left $\cO(G)$-covariant differential $\ast$-calculus
\begin{align*}
\Om^{\bullet}_q(G/L_S) \in  \mbox{}^{~~~ \cO_q(G)}_{\cO_q(G/L_S)}\mathrm{Mod}_{0},
\end{align*} 
of classical dimension
\begin{align*}
\mathrm{dim} \, \Phi \big( \Om^{k}_q(G/L_S) \big) = \left( \begin{array}{c} \!\! 2M  \!\! \\  \!\! k  \!\! \end{array} \right), ~~~~  \mbox{for all} ~~ k = 1, \ldots, 2M, 
\end{align*} 
where $M$ is the complex dimension of the corresponding classical manifold. Moreover
\begin{enumerate}[(i)]
\item $\Om^{\bullet}_q(G/L_S)$ admits a unique left $\cO_q(G)$-covariant complex structure
\begin{align*}
\Om^{\bullet}_q(G/L_S) \simeq \bigoplus_{ (a,b) \in \mathbb{N}_0 } \Om^{(a,b)} =: \Om^{(\bullet,\bullet)}, 
\end{align*}
\item $\Om^{(1,0)}$ and $\Om^{(0,1)}$ are irreducible objects in $\mbox{}^{~~~ \cO_q(G)}_{\cO_q(G/L_S)}\mathrm{Mod}_0$.
\end{enumerate}
\end{thm}
The part (i) of the Theorem \ref{HKcalc} is shown in \cite{MM} where the $\ast$-structure on the generators of $\cO_q(G/L_S)$ is given by $(z_{ij})^* = z_{ji}$.
The following theorem gives us the existence of the K\"ahler structure for the Heckenberger--Kolb calculus.

\begin{thm}(\cite[Theorem 5.10]{MM}).
Let $\Om^{\bullet}_q(G/L_S)$ be the Heckenberger--Kolb calculus of the irreducible quantum flag manifold $\cO_q(G/L_S)$. Then there exists a form $\kappa \in \Om^{(1,1)}$ such that $(\Om^{(\bullet, \bullet)}, \kappa)$ is a covariant K\"ahler structure for all $q \in \mathbb{R}_{>0}\setminus F$, where $F$ is a finite, possibly empty, subset of $\mathbb{R}_{>0}$. Moreover, any element of $F$ is necessarily non-transcendental.
\end{thm}

It has been shown in \cite[\textsection 5.7]{RP} that if $\cO_q(G)$ is a compact quantum group algebra and $q$ belongs to a suitable open interval around $1$ then an inner product on $\Om^{(\bullet, \bullet)}$ can be defined by
\begin{align}\label{innerprodc}
\langle \cdot , \cdot \rangle : \Om^{(\bullet, \bullet)} \times \Om^{(\bullet, \bullet)} \rightarrow \mathbb{C}, ~~~~~~  (\omega, \nu) \mapsto \mbox{\textbf{h}} \circ g_{\sigma}(\omega, \nu),
\end{align}
where $\textbf{h}$ is the Haar state associated to $\cO_q(G)$ \cite[\textsection 11.3.2]{KS}. From now on we consider $q\in (1,\epsilon)$ so that the inner product is well defined. With respect to this inner product we consider the adjoints $\mathrm{d}^{\dagger}$, $\del^{\dagger}$ and $\adel^{\dagger}$. It is shown in \cite[\textsection 5.4]{RO} that
\begin{align}\label{adjoints}
\mathrm{d}^{\dagger} = -\ast_{\kappa} \circ \, \mathrm{d} \circ \ast_{\kappa}, ~~~~~~~~ \del^{\dagger} = -\ast_{\kappa}\circ \, \adel \circ \ast_{\kappa}, ~~~~~~~~ \adel^{\dagger} = -\ast_{\kappa} \circ \, \del \circ \ast_{\kappa}.
\end{align} 
Then Theorem \ref{HKcalc} and the inner product \eqref{innerprodc} allow us to define the $\adel$-Dirac operator and the $\adel$-Laplacian 
\begin{align*}
D_{\overline{\partial}} := \overline{\partial} + \overline{\partial}^{\dagger}, ~~~~~~~~~
 \Delta_{\overline{\partial}}:=( \overline{\partial} + \overline{\partial}^{\dagger})^2. 
\end{align*}

Let $I$ be the set $\{1, \ldots, N:= \mathrm{dim}(V_{\varpi_x})\}$, it follows from Proposition 3.3 and Proposition 3.4 in \cite{HK} that $\Phi(\Om^{(1,0)})$ and $\Phi(\Om^{(0,1)})$ is generated as a vector space by $\{[\partial z_{iN}] \ | \ i\in I_{(1)} \}$ and $\{[\overline{\partial}z_{Ni}] \ | \ i \in I_{(1)}\}$ respectively, where $I_{(1)}= \{ i\in I: (\varpi_x, \varpi_x -\alpha_x -\mbox{wt}(v_i))=0 \}$.

\begin{lem}\label{Sphewh}
For the Heckenberger-Kolb calculus of $\cO_q(\mathbf{Q}_N)$ it holds that 
\begin{enumerate}[(i)]
\item $\Phi(\Omega^{(0,1)}) = \Omega^{(0,1)}/\cO_q(\mathbf{Q}_N)^{+} \Omega^{(0,1)} =  \mathrm{Lin}_{\mathbb{C}} \big\{ [\overline{\partial}z_{Ni}]: i = 2,\ldots, N-1 \big\}$ and 
$\Phi(\Omega^{(1,0)}) = \Omega^{(1,0)}/\cO_q(\mathbf{Q}_N)^{+} \Omega^{(1,0)} = \mathrm{Lin}_{\mathbb{C}} \big\{ [\partial z_{iN}]: i = 2,\ldots, N-1 \big\}$

\item if we consider $\mathcal{O}_q(G/L_S)$ as a left $U_q(\mathfrak{g})$-module with action $X \btr z := z \ltr S(X)$ then the generators of $ \mathcal{O}_q(\mathrm{Q}_{N})$ are given by $z:= u^1_1u^1_N$, $y := u^1_1u^2_N - qu^{2}_1u^1_N$ with weights $2\varpi_1$, $\varpi_2$, respectively.
\end{enumerate}
\end{lem}

\begin{proof}
(i)  Since $v_N$ is the highest weight vector of $V_{\varpi_1}$ we have $(\varpi_1, \varpi_1 -\alpha_1 -\mbox{wt}(v_N))= (\varpi_1, \varpi_1 -\alpha_1 - \varpi_1) = -(\varpi_1, \alpha_1) \neq 0$, thus $v_N \notin I_{(1)}$. Similarly, if we consider $v_1$ as the lowest weight vector, then $\mbox{wt}(v_1) = w_0(\varpi_1)= -\mbox{id}(\varpi_1)= -\varpi_1$. Then $(\varpi_1, \varpi_1 -\alpha_1 -\mbox{wt}(v_1))= (\varpi_1, \varpi_1 -\alpha_1 +\varpi_1)= 2(\varpi_1,\varpi_1)-(\varpi_1, \alpha_1) = (2r_1-1)(\varpi_1, \alpha_1)$ where $r_1$ is the coefficient of $\alpha_1$ in the expression of $\varpi_1$ as linear combination of the $\alpha_1, \ldots , \alpha_n$. From the form of the Cartan matrix of type $B_n$ and $D_n$ it can be shown that $r_1 =1$ and therefore $v_1 \notin I_{(1)}$. It also can be shown that $M= N-2$ for both cases $\cO_q(\mathbf{Q}_{2n+1})$ and $\cO_q(\mathbf{Q}_{2n})$, see for example \cite[Table 2]{RP}. Therefore $I_{(1)}=\{2, \ldots, N-1\}$.

For the proof of (ii) we first note that the left and right actions on matrix coefficients are given by
\begin{align}\label{leftright}
(X \triangleright c^{\varpi_1}_{f_i,v_j} \triangleleft Y)(Z):= f_i(YZX \triangleright v_j) = c^{\varpi_1}_{f_i \triangleleft Y, X \triangleright v_j}(Z),  
\end{align}
for all $X, Y, Z\in U_q(\mathfrak{g})$. From \ref{leftright} or identities \eqref{actEiodd}-\eqref{actEieven} we have that $z,y \in \cO_q(\mathbf{Q}_N)$.
By \eqref{leftright} we have $E_i \btr (u^1_1 u^1_N) = (u^1_1 u^1_N) \triangleleft S(E_i) = -(c^{\varpi_1}_{f_1 \triangleleft E_i ,v_1} c^{\varpi_1}_{f_1 \triangleleft K_i, v_1} + c^{\varpi_1}_{f_1,v_1}c^{\varpi_1}_{f_1 \triangleleft E_i ,v_N}) \ltr K_i^{-1} = 0$ since $f_1$ is the highest weight of the dual representation. On the other hand, 
\begin{align*}
K_i \btr (u^1_1 u^1_N) = (u^1_1 u^1_N) \triangleleft K^{-1}_i = c^{\varpi_1}_{f_1 \triangleleft K^{-1}_i,v_j} c^{\varpi_1}_{f_1 \triangleleft K^{-1}_i ,v_N}= q^{(2\varpi_1, \alpha_i)}u^1_1 u^1_N. 
\end{align*}

Similarly, since $f_1$ is a highest weight vector we have for $i\neq 1$ 
\begin{align}\label{hwE}
(u^1_1 u^{2}_N -q u^{2}_1 u^1_N) \ltr E_i = c^{\varpi_1}_{f_1,v_1} c^{\varpi_1}_{f_2 \ltr E_i ,v_N} - q c^{\varpi_1}_{f_2 \ltr E_i, v_1} c^{\varpi_1}_{f_1 \ltr K_i, v_N}.  
\end{align}
However, since $v_1$ is the lowest weight vector of $V_{\varpi_1}$ with weight $w_0 \varpi_1 = -\varpi_1$ then  for $i \neq 1$ the vector $v_1$ is the lowest weight with weight zero of the $U_q(\mathfrak{sl}_2)_i$-module $U_q(\mathfrak{sl}_2)_i v_N$ where $U_q(\mathfrak{sl}_2)_i := \langle E_i, F_i, K_i \rangle$. Therefore $E_i v_1 = 0$ for $i\neq 1$ and $E_1 v_1= v_2$. This in turn implies that $f_2 \ltr E_i =0$ for $i\neq 1$. Then the right hand side of equation \eqref{hwE} is zero. Since $f_2 \ltr E_1 =f_1$ and $f_1 \ltr K_1 = f_1$ we have
\begin{align*}
E_1 \btr (u^1_1 u^{2}_N -q u^{2}_1 u^1_N) & = -(u^1_1 u^{2}_N -q u^{2}_1 u^1_N) \ltr E_1K_1^{-1} \\
&= -(c^{\varpi_1}_{f_1,v_1} c^{\varpi_1}_{f_2 \ltr E_1 ,v_N} - q c^{\varpi_1}_{f_2 \ltr E_1, v_1} c^{\varpi_1}_{f_1 \ltr K_1, v_N}) \ltr K_i^{-1} \\
& =0.  
\end{align*}
Finally, since $K_i E_1 v_1 = q^{(\alpha_i, \alpha_1)} E_1 K_i v_1 = q^{(\alpha_i, \alpha_1) - (\varpi_1, \alpha_i)} E_1 v_1 $ we have
\begin{align*}
K_i \btr (u^1_1 u^{2}_N -q u^{2}_1 u^1_N) &= (u^1_1 u^{2}_N -q u^{2}_1 u^1_N) \ltr K^{-1}_i \\
&= c^{\varpi_1}_{f_1 \ltr K^{-1}_i,v_1} c^{\varpi_1}_{f_2 \ltr K^{-1}_i ,v_N} - q c^{\varpi_1}_{f_2 \ltr K^{-1}_i, v_1} c^{\varpi_1}_{f_1 \ltr K^{-1}_i, v_N} \\
& = q^{( 2\varpi_1 - \alpha_1 , \alpha_i)} (c^{\varpi_1}_{f_1, v_1} c^{\varpi_1}_{f_2 ,v_N} - q c^{\varpi_1}_{f_2, v_1} c^{\varpi_1}_{f_1, v_N}) \\
&= q^{(\varpi_2 , \alpha_i)}(u^1_1 u^{2}_N - q u^{2}_1 u^1_N),  
\end{align*}
where we used the fact that $\alpha_i = \sum_j a_{ji} \varpi_j$ to derive the third equality. From this calculations we have that $z$ and $y$ are highest weight vectors with the correct weights. According to Table \ref{sphericalweights} of the appendix \ref{tablesph} we have that the left $U_q(\mathfrak{so}_N)$-module generated by $\{z,y\}$ is $\cO_q(\mathbf{Q}_N)$. 
\end{proof}


\begin{lem}\label{adelzzadel}
Let $\partial$, $\overline{\partial}$, $z,y$ be as before. Then the following relations hold
\begin{enumerate}[(i)]
\item $\overline{\partial} z z= q^{-2} z\overline{\partial} z$,
\item $\overline{\partial} y y= q^{-2} y\overline{\partial} y$,
\item $\partial z z= q^{2} z\partial z$,
\item $\partial y y= q^{2} y\partial y$.
\end{enumerate}
\end{lem}
\begin{proof}
We only show (i) and (ii), equations (iii) and (iv) can be shown in a similar way by using Lemma \ref{Rels} and Lemma \ref{Sphewh}. By Takeuchi's equivalence there exists an isomorphism in the category ${}^{\cO_q(SO(N))}_{~~~ \cO_q(Q_N)}\mathrm{Mod}_0$ given by
\begin{align*} 
\rU : \Omega^{(0,1)} \rightarrow  \mathcal{O}_q(SO(N)) \square_{\cO_q(L_S)} \Phi(\Omega^{(0,1)}), \qquad \omega \mapsto \omega_{(-1)} \otimes [\omega_{(0)}]. 
\end{align*}
Similar isomorphism exists for $\Omega^{(1,0)}$. Now we recall by Lemma \ref{Sphewh} that $\{[\adel z_{Nj}]: j=2,\ldots,N-1\}$ and  $\{[\del z_{jN}]$, $j=2,\ldots, N-1\}$ are bases of the vector spaces $\Phi(\Om^{(0,1)})$ and $\Phi(\Om^{(1,0)})$, respectively.
Moreover, it was shown in \cite[\textsection 3.2]{HK} that $[\adel z_{ij}]=[\del z_{ji}]=0$ if $i\neq N$ and $j \in \{1,N\}$.

Then by definition of $\rU$ and relations in Lemma \ref{Rels} we have
\begin{align*}
\rU \big( (\adel z_{1N})z_{1N} \big) &= \rU (\adel z_{1N})z_{1N} \\
&= (z_{1N})_{(1)}z_{1N} \otimes [ \adel (z_{1N})_{(2)} ] \\
&= (u^1_N S(u^N_N))_{(1)} z_{1N} \otimes [\adel (u^1_N S(u^N_N))_{(2)} ] \\
&= \sum_{j,k} u^1_j S(u^k_N) z_{1N} \otimes [\adel (u^j_N S(u^N_k)) ] \\
&= \sum_{j,k} u^1_j S(u^k_N) z_{1N} \otimes [\adel z_{jk}] \\
&= \sum_{k=2}^{N-1} u^1_N S(u^k_N) z_{1N} \otimes [\adel z_{Nk}]\\
&= \sum_{k=2}^{N-1} u^1_N S(u^k_N)u^1_N S(u^N_N) \otimes [\adel z_{Nk}] \\
&= \sum_{k=2}^{N-1} q^{\rho_N - \rho_k}u^1_N u^1_{k'}u^1_N u^1_1 \otimes [\adel z_{Nk}] \\
&= q^{-2} \sum_{k=2}^{N-1} q^{\rho_N - \rho_k} u^1_N u^1_1 u^1_N u^1_{k'} \otimes [\adel z_{Nk}] \\
&= q^{-2} z_{1N} \rU(\adel z_{1N}).
\end{align*}
This yields equation (i).


Since $y = q(z_{2, N} - q^{-2} z_{1, N-1})$ then analogously by using relations of Lemma \ref{Rels} we have  
\begin{align*}
\rU \big( (\adel y ) y \big) &= \rU (\adel y ) y \\
&= \rU(\adel (qz_{2,N} - q^{-1}z_{1,N-1}))(qz_{2,N} - q^{-1}z_{1,N-1}) \\
&= ( qz_{2,N} - q^{-1}z_{1,N-1} )_{(1)} ( qz_{2,N} - q^{-1}z_{1,N-1}) \otimes [ \adel (q z_{2,N} - q^{-1}z_{1,N-1} )_{(2)} ] \\
& = \sum_{l=2}^{N-1} ( qu^2_N S(u^l_N) - q^{-1} u^1_N S(u^l_{N-1}) )(u^1_1u^2_N - qu^2_1u^1_N) \otimes [ \adel (qz_{Nl} - q^{-1} z_{Nl})] \\
&= \sum_{l=2}^{N-1} q^{\rho_{N}- \rho_l }( qu^2_N u^1_{l'} - u^1_N u^2_{l'})(u^1_1u^2_N - qu^2_1u^1_N) \otimes [ \adel (qz_{Nl} - q^{-1} z_{Nl})] \\
&= \sum_{l=2}^{N-1} q^{\rho_{N}- \rho_l +1 }( u^1_{l'} u^2_N - q u^2_{l'} u^1_N )(u^1_1u^2_N - qu^2_1u^1_N) \otimes [ \adel (qz_{Nl} - q^{-1} z_{Nl})] \\
&= q^{-2}\sum_{l=2}^{N-1} q^{\rho_{N}- \rho_l + 1 }(u^1_1u^2_N - qu^2_1u^1_N) ( u^1_{l'} u^2_N - q u^2_{l'} u^1_N ) \otimes [ \adel (qz_{Nl} - q^{-1} z_{Nl})] \\
&= q^{-2}(u^1_1u^2_N - qu^2_1u^1_N) \sum_{l=2}^{N-1} q^{\rho_{N}- \rho_l + 1 } ( u^1_{l'} u^2_N - q u^2_{l'} u^1_N ) \otimes [ \adel (qz_{Nl} - q^{-1} z_{Nl})] \\
&= q^{-2} y \rU(\adel y).
\end{align*}

This proves equation (ii).

\end{proof}

In the following lemma we give the existence of a K\"ahler structure for the Heckenberger--Kolb calculus over the quantum quadric $\cO_q(\textbf{Q}_N)$. Since the space of coinvariants $\mbox{}^{co(L_S)}\Phi(\Om^{(1,1)})$ is a 1-dimensional vector space (see for example \cite[\textsection 3.3]{FRK}) it can be shown that this K\"ahler form is up to scalar equal to the K\"ahler form given in \cite[\textsection 5]{MM}. The following lemma is proven in \cite{MM} in a general context, but for the convenience of the reader we give a different proof for $\cO_q(\mathbf{Q}_N)$ using basic computations.

\begin{lem}\label{Kahlerform}
A K\"ahler form is given by
\begin{align*}
\kappa = \mathrm{i} \sum_{i,j=1}^N q^{-\rho_i} \del z_{ij} \wedge \adel z_{ji}.
\end{align*}
For $i=1,\ldots,N-2$ we denote $e^{+}_i:= [\del z_{i+1,N}]$ and $e^{-}_{i}:= q^{-\rho_{i+1}}[\adel z_{N,i+1}]$, then 
\begin{align}\label{Kahlers}
\rU(\kappa)= \mathrm{i} \sum_{i=1}^{N-2} 1\otimes e^{+}_i \wedge e^{-}_i 
\end{align}
and there exist polynomials $f_{I,J}(q) \in \mathbb{C}[q,q^{-1}]$ such that 
\begin{align}\label{powerskap}
\rU(\kappa^{l})= \mathrm{i}^{l} \sum_{I,J \in \Theta(l)} f_{l,I,J}(q)\!\cdot \! 1\otimes e^{+}_I \wedge e^{-}_J, \quad   \ \ f_{l,I,I}(1)\neq 0, \ \ \text{and} \ \ f_{l,I,J}(1)=0  \ \ \text{for} \ \ I\neq J,  
\end{align} 
where $\Theta(l)$ denotes the set of all ordered subsets of $\{1, \ldots, N-2 \}$ with $l$ elements. Moreover, sgn$(f_{l,I,I}(1))=(-1)^{\frac{l(l-1)}{2}}$ for all $I\in \Theta(l)$.
\end{lem}
\begin{proof}
That $\kappa$ is real follows from identities \eqref{dif.inv.} and the fact that $z_{ij}^{*} = z_{ji}$.
Equation \eqref{Kahlers} follows from the calculation
\begin{align*}
\rU \Big( \sum_{i,j=1}^N q^{-2\rho_i} \del z_{ij} \wedge \adel z_{ji} \Big) & \!= \sum_{i,j=1}^{N} q^{-2\rho_i}(z_{ij})_{(1)}(z_{ji})_{(1)} \otimes [\del ( (z_{ij})_{(2)} )] \wedge [\adel ((z_{ji})_{(2)}) ] \\
& \!= \sum_{i,j=1}^{N} \sum_{a,b,c,d }q^{-2\rho_i} u^i_a S(u^b_j) u^j_c S(u^d_i) \otimes [\del (u^a_N S(u^N_b)) ] \wedge [ \adel ( u^c_N S(u^N_d)) ] \\
& \!= \sum_{i=1}^{N} \sum_{a,b,d } q^{-2\rho_i} u^i_a S(u^d_i) \otimes [ \del (u^a_N S(u^N_b))] \wedge [\adel ( u^b_N S(u^N_d))] \\
& \!= \sum_{i=1}^{N} \sum_{a,d} q^{-2\rho_i} u^i_a S(u^d_i) \otimes [\del z_{aN}] \wedge [\adel z_{Nd}] \\
& \!= \sum_{d=2}^{N-1} 1 \otimes [\del z_{dN}] \wedge q^{-\rho_d}[\adel z_{Nd}],
\end{align*}
where the identity $\sum_{i=1}^N q^{-2\rho_i}u^i_a S(u^d_i) = q^{-2\rho_d} \delta_{da}\! \cdot \!1$ was used in the penultimate line. This implies the left $\cO_q(G)$-coinvariance of $\kappa$ since 
\begin{align*}
\Big[ \sum_{i,j=1}^N q^{-2\rho_i}\del z_{ij} \wedge \adel z_{ji} \Big] = \sum_{i,j=1}^N q^{-2\rho_i}[\del z_{ij}] \wedge [\adel z_{ji}]  = \sum_{i=2}^{N-1} [\del z_{iN}] \wedge q^{-2\rho_i}[\adel z_{Ni}].  
\end{align*}
Now we show that $\kappa$ is closed by a Lie theoretic argument: we set $V^{(a,b)}:= \Phi(\Om^{(a,b)})$ and note that $\mathrm{d}\kappa =0$ if and only if $\del \kappa= \adel \kappa =0$, and both are coinvariant elements in $\Omega^{(2,1)}$ and $\Omega^{(1,2)}$, respectively. But $\mbox{}^{co(L_S)}V^{(1,2)} = \mbox{}^{co(L_S)}(V^{(1,0)} \otimes V^{(0,2)})$, then $\adel \kappa=0$ if $\mbox{}^{co(L_S)}(V^{(1,0)} \otimes V^{(0,2)})=0$. Since $V^{(1,0)} \otimes W$ contains the trivial representation if and only if $W \simeq (V^{(1,0)})^{*} \simeq V^{(0,1)}$ it is enough to show that $V^{(0,2)}$ is not isomorphic to $V^{(1,0)}$. But $V^{(0,2)}$ is an irreducible component of $V^{(0,1)} \otimes V^{(0,1)}$. The following formulas for the tensor product decomposition can be found in Table 5 of \cite{OV}: for odd $N= 2n+1$ we have
\begin{align*}
V^{(0,1)} \otimes V^{(0,1)} \simeq \left\{ \begin{array}{lc}
V_{4\varpi_1} \oplus V_{2\varpi_1} \oplus \mathbb{C}, & n=2 \ \ ( V^{(0,1)} \simeq V_{2\varpi_1} \ \text{as} \  U_q(\mathfrak{sl}_2)\text{-modules} ),  \\  
V_{2\varpi_1} \oplus V_{2\varpi_2} \oplus \mathbb{C}, & n=3 \ \ (V^{(0,1)} \simeq V_{\varpi_1} \ \text{as} \  U_q(\mathfrak{so}_{5})\text{-modules}),  \\
V_{2\varpi_1} \oplus V_{\varpi_2} \oplus \mathbb{C}, & ~~~~ n\geq 4  \ \ (V^{(0,1)} \simeq V_{\varpi_1} \ \text{as} \  U_q(\mathfrak{so}_{2n-1})\text{-modules}).
\end{array}\right.
\end{align*}
For $n=2$, it can be shown that the $U_q(\mathfrak{l}^s_S)$-invariant element corresponding to $\mathbb{C}$ is not $U_q(\mathfrak{l}_S)$-invariant. For even $N=2n$ we have
\begin{align*}
V^{(0,1)} \otimes V^{(0,1)} \simeq \left\{ \begin{array}{lc} 
V_{2\varpi_2} \oplus V_{\varpi_3 + \varpi_1} \oplus \mathbb{C}, & n=4  \ \ (V^{(0,1)} \simeq V_{\varpi_2} \ \text{as} \ \ U_q(\mathfrak{sl}_4)\text{-modules} ),  \\
V_{2\varpi_1} \oplus V_{\varpi_2} \oplus \mathbb{C}, & ~~~~~ n\geq 5 \ \ (V^{(0,1)} \simeq V_{\varpi_1} \ \text{as} \ \ U_q(\mathfrak{so}_{2n-2})\text{-modules}).
\end{array}\right.
\end{align*} 
This formulas show that $V^{(1,0)}$ does not appear as an irreducible component of $V^{(0,1)} \otimes V^{(0,1)}$ and therefore $\adel \kappa =0$. A similar argument shows that $\del k=0$. By Lemma 4.6 in \cite{RO} we have that $\kappa$ is central. The bijectivity of the Lefschetz operators $L^{n-k}$ is shown in \cite{MM} for $q$ belonging to a sufficiently small open interval around $1$. 

Now we proceed to calculate the powers of $\kappa$ by using the relations in $\Phi(\Om^{(\bullet, 0)})$ and $\Phi(\Om^{(0,\bullet)})$. It follows from \cite{HK} that $\Phi(\Om^{(0,1)})= \mathrm{Lin}_{\mathbb{C}} \{e^{-}_1, \ldots, e^{-}_{N-2}\}$ and the algebra $\Phi(\Om^{(0,\bullet)})$ is isomorphic to the quantum exterior algebra $\Lambda(\cO^{N-2}_q)$ of the quantum Euclidean space $\cO^{N-2}_q$ in \cite[\textsection 9.3]{KS}. Then the relations in \cite{HK} imply that the relations in $\Phi(\Om^{(0, \bullet)})$ for $N=2n+1$ are given by
\begin{align}
e^{-}_i \wedge e^{-}_i = 0, \ \  i\neq i',  ~~~~~~ e^{-}_i \wedge e^{-}_j = -q^{-1} e^{-}_j \wedge e^{-}_i , \ \ i<j, \ i\neq j', \label{ii}\\
e^-_{i'} \wedge e^-_i + e^-_i \wedge e^-_{i'} = (q-q^{-1}) \sum_{1\leq j <i} \lambda_i^{-1}\lambda_j q^{j-i+1}e^-_j \wedge e^-_{j'}, \ i <i', ~~ \label{i'i} \\
e^{-}_{n} \wedge e^{-}_{n} = (q^{1/2} -q^{-1/2}) \sum_{1\leq j \leq n-1} \lambda_{n}^{-1}\lambda_j q^{j-(n-1)}e^{-}_j \wedge e^{-}_{j'}, ~~~~~~~~ \label{nn}
\end{align}
where $\lambda_i$ is a non-zero complex number. For even $N=2n$ the relations are given by $\eqref{ii}$ and $\eqref{i'i}$ for possibly different constants $\lambda_i$. The relations for $\Phi(\Om^{(\bullet,0)})$ are giving by applying the involution to the relations \eqref{ii}-\eqref{nn}. If we apply the involution to \eqref{i'i} we get 
\begin{align}
e^+_{i} \wedge e^+_{i'} + e^+_{i'} \wedge e^+_{i} = (q-q^{-1}) \sum_{1 \leq j <i} \overline{\lambda}_i^{-1} \overline{\lambda}_j q^{j-i+1}e^+_{j'} \wedge e^+_{j}, \ \ i <i'. \label{invi'i}
\end{align}
We claim that the relations in $\Phi(\Om^{(\bullet, 0)})$ are given by 
\begin{align}
e^{+}_i \wedge e^{+}_i = 0, \ \  i\neq i', ~~~~~~  e^{+}_i \wedge e^{+}_j = -q e^{+}_j \wedge e^{+}_i , \ \ i<j, \ i\neq j', \label{holii}\\
e^+_{i'} \wedge e^+_i + e^+_i \wedge e^+_{i'} = -(q-q^{-1}) \sum_{1\leq j <i} \overline{\lambda}_i^{-1}\overline{\lambda}_j q^{-(j-i+1)}e^+_j \wedge e^+_{j'}, \ i <i', ~~ \label{holi'i}\\
e^{+}_{n} \wedge e^{+}_{n} = -(q^{1/2} -q^{-1/2}) \sum_{1\leq j \leq n-1} \overline{\lambda}_n^{-1} \overline{\lambda}_j q^{-(j-(n-1))}e^{+}_j \wedge e^{+}_{j'}. ~~~~~~ \label{holnn}
\end{align}
Since the involution $\ast: \Omega^{\bullet} \rightarrow \Omega^{\bullet}$ of the calculus induces an involution $\bar{\ast}: \Phi(\Omega^{\bullet}) \rightarrow \Phi(\Omega^{\bullet})$ we get the relations \eqref{holii} by applying this involution to relations \eqref{ii}. For the rest of the paper we denote $\nu_{\tau}:= \tau-\tau^{-1}$ if $\tau>0$. We use induction to prove \eqref{holi'i}:  
\begin{align*}
\overline{\lambda}_{i+1}\big( & e^+_{i+1}  \wedge e^+_{(i+1)'}  + e^+_{(i+1)'} \wedge e^+_{i+1} \big) =  \nu_q \sum_{1 \leq j <i+1} \overline{\lambda}_{j} q^{j-(i+1)+1}e^+_{j'} \wedge e^+_{j}  ~~~~~~~~~~~~~~~~~~ \mathrm{by} \ \ \eqref{invi'i} \\
&= q^{-i} \nu_{q} \sum_{1 \leq j <i+1} \overline{\lambda}_j q^{j} \big(\! -\nu_q \sum_{1\leq k < j}\overline{\lambda}_{j}^{-1}\overline{\lambda}_k q^{-k+j-1}e^+_k \wedge e^+_{k'} - e^+_j \wedge e^+_{j'} \big) ~~~~~~~~~  (\text{ind. hyp.}) \\
&= q^{-i} \nu_q \Big(\!\! -\nu_q \sum_{1 \leq j <i+1} \sum_{1\leq k < j} \overline{\lambda}_k q^{2j-k-1} e^+_k \wedge e^+_{k'} -  \sum_{1 \leq j <i+1} \overline{\lambda}_j q^{j} e^+_j \wedge e^+_{j'} \Big) \\
&= q^{-i} \nu_q \Big(\!\! -\nu_q  \sum_{1 \leq k <i}  \sum_{k+1 \leq j \leq i} \overline{\lambda}_k q^{2j-k-1} e^+_k \wedge e^+_{k'} -  \sum_{1 \leq j <i+1} \overline{\lambda}_j q^{j} e^+_j \wedge e^+_{j'} \Big) \\
&= q^{-i} \nu_q \Big(\!\! -\nu_q \sum_{1 \leq k <i} \overline{\lambda}_k q^{-k-1}\Big(\frac{q^{2(k+1)}-q^{2i+2}}{1-q^2} \Big) e^+_k \wedge e^+_{k'} -  \sum_{1 \leq j <i+1} \overline{\lambda}_j q^{j} e^+_j \wedge e^+_{j'}  \Big) \\
&= q^{-i} \nu_q \Big(  \sum_{1 \leq k <i} \overline{\lambda}_k ( q^{k}-q^{2i-k} ) e^+_k \wedge e^+_{k'} -  \sum_{1 \leq j <i+1} \overline{\lambda}_j q^{j} e^+_j \wedge e^+_{j'} \Big) \\
&= q^{-i} \nu_q \Big( \sum_{1 \leq k <i} \overline{\lambda}_k q^{k}e^+_k \wedge e^+_{k'} - \sum_{1 \leq k <i} \overline{\lambda}_k q^{2i-k} e^+_k \wedge e^+_{k'} -  \sum_{1 \leq j <i+1} \overline{\lambda}_j q^{j} e^+_j \wedge e^+_{j'} \Big) \\
&= \nu_q \Big(  - \sum_{1 \leq k <i} \overline{\lambda}_k q^{i-k} e^+_k \wedge e^+_{k'} -  \overline{\lambda}_i e^+_i \wedge e^+_{i'}  \Big).
\end{align*}
Similar calculations show the identity \eqref{holnn} by using \eqref{holi'i}.

Now assuming that the equation \eqref{powerskap} holds for $l$, we have
\begin{align*}
\rU(\kappa^{l+1}) &= \rU(\kappa) \wedge \mathrm{i}^{l} \sum_{I,J \in \Theta(l)}f_{I,J}(q)\! \cdot \! 1 \otimes e^{+}_I \wedge e^{-}_J \\
&=  \mathrm{i}^{l+1} \sum_{I,J \in \Theta(l)} \sum_i f_{I,J}(q)\! \cdot \! 1 \otimes e^{+}_I \wedge e^{+}_i \wedge e^{-}_i \wedge e^{-}_J,
\end{align*}
where we use the fact that $\kappa$ is central. First we can focus on the case $e^{+}_I \wedge e^{+}_i \wedge e^{-}_i \wedge e^{-}_I$ for $I=\{i_1,\ldots, i_l\}$ fixed. Then different cases arise:\\
\\
\textbf{Case} $i\neq i'$ and $i\in I$: $e^{+}_I \wedge e^{+}_i \wedge e^{-}_i \wedge e^{-}_I=0$  since we can move $e^{-}_i$ to the right(resp. left) by using \eqref{ii}(resp. \eqref{holii}) if $i<i'$(resp. $i>i'$).
\\
\\
\textbf{Case} $i, i' \notin I$: in this case we can move $e^{+}_i$ to the left and $e^{-}_i$ to the right by using \eqref{holii} and \eqref{ii}, respectively and we get
\begin{align*}
e^{+}_I \wedge e^{+}_i \wedge e^{-}_i \wedge e^{-}_I = (-1)^{l}q^{l-2\alpha_i} e^{+}_{i_1}\wedge \cdots \wedge e^{+}_{i} \wedge \cdots \wedge e^{+}_l \wedge e^{-}_{i_1}\wedge \cdots \wedge e^{-}_{i} \wedge \cdots \wedge e^{-}_l,
\end{align*}
where $\alpha_i\in \mathbb{N}_0$.
\\
\\
\textbf{Case} there is a $k<l$ such that $i_{k}=i' \geq i$ and $i\notin I$: in this case we can move $e^{+}_i$ to the left by commuting with $e^{+}_{i_l}, \ldots, e^{+}_{i_{k+1}}$ by using relations \eqref{holii} and we get
\begin{align*}
e^{+}_{i_1} \wedge \cdots \wedge e^{+}_{i_l} & \wedge e^{+}_{i} \wedge e^{-}_i \wedge e^{-}_{i_1} \wedge \cdots \wedge e^{-}_{i_l} \\
& = \! (-q)^{-(l-k)} e^{+}_{i_1} \wedge \cdots \wedge e^{+}_{i_{k}}  \wedge e^{+}_i \wedge \cdots \wedge e^{+}_l \wedge e^{-}_i \wedge e^{-}_{i_1} \wedge \cdots \wedge e^{-}_{i_l}. 
\end{align*}
Now if $i_k = i' >i$ then we can use \eqref{holi'i} with $\gamma_{ij}:= -\lambda^{-1}_i \lambda_j $ and we get
\begin{align*}
& e^{+}_I  \wedge e^{+}_i \wedge e^{-}_i \wedge e^{-}_I = \! (-q)^{-(l-k)} e^{+}_{i_1} \wedge \cdots \wedge e^{+}_{i_{k}}  \wedge e^{+}_i \wedge \cdots \wedge e^{+}_{i_l} \wedge e^{-}_i \wedge e^{-}_{i_1} \wedge \cdots \wedge e^{-}_{i_l} \\
&= \! (-q)^{-(l-k)} \nu_q \! \sum_{1\leq j <i} \overline{\gamma}_{ij}q^{-(j-i+1)}e^{+}_{i_1} \wedge \cdots \wedge e^{+}_{i_{k-1}} \wedge e^{+}_j \wedge e^{+}_{j'} \wedge \cdots \wedge e^{+}_{i_l} \wedge e^{-}_i \wedge e^{-}_{i_1} \wedge \cdots \wedge e^{-}_{i_l}   \\
&\quad - (-q)^{-(l-k)} e^{+}_{i_1} \wedge \cdots \wedge e^{+}_{i_{k-1}} \wedge  e^{+}_i \wedge e^{+}_{i_k} \wedge \cdots \wedge e^{+}_{i_l} \wedge e^{-}_i  \wedge e^{-}_{i_1} \wedge \cdots \wedge e^{-}_{i_l}.  
\end{align*}
Since $j' \neq i_{r}$ and $i'\neq i_r$ for $r\neq k$ then we can move $e^{+}_j$ and $e^{+}_i$ to the left and $e^{+}_{j'}$ and $e^{-}_i$ to the right and we obtain
\begin{align*}
& e^{+}_I  \wedge e^{+}_i \wedge e^{-}_i \wedge e^{-}_I \\
&= \! (-q)^{-(l-k)} \nu_q \!\! \sum_{1\leq j_s <i} \! \overline{\gamma}_{ij}q^{-(j_s-i+1)+\alpha_{is}}e^{+}_{i_1} \wedge \! \cdots \! \wedge e^{+}_{j_s} \wedge \! \cdots \! \wedge e^{+}_{j'_s} \wedge \! \cdots \! \wedge e^{+}_{i_l} \wedge e^{-}_{i_1} \wedge \! \cdots \! \wedge e^{-}_i \! \wedge \! \cdots \! \wedge e^{-}_{i_l}   \\
&\quad +(-1)^{l}q^{1-l+2\beta_i} e^{+}_{i_1} \wedge \cdots \wedge  e^{+}_i \wedge \cdots \wedge e^{+}_{i_k} \wedge \cdots \wedge e^{+}_{i_l}  \wedge e^{-}_{i_1} \wedge \cdots  \wedge e^{-}_i \wedge \cdots \wedge e^{-}_{i_k} \wedge \cdots \wedge e^{-}_{i_l},
\end{align*}
where $\alpha_{is}\in \mathbb{Z}$ and $\beta_i \in \mathbb{N}_0$. Now if $i=i'=n$ and $n\in I$ then by \eqref{i'i} and \eqref{holi'i} we have
\begin{align*}
& e^{+}_I  \wedge e^{+}_n \wedge e^{-}_n \wedge e^{-}_I = \!(-1)^{l-1} q^{l-1-2\gamma_n} e^{+}_{i_1} \wedge \! \cdots \! \wedge e^{+}_{n}  \wedge e^{+}_n \wedge \! \cdots \! \wedge e^{+}_{i_l}  \wedge e^{-}_{i_1} \wedge \! \cdots \! \wedge e^{-}_n \wedge e^{-}_n \wedge \! \cdots \! \wedge e^{-}_{i_l} \\
&= \! (-1)^l q^{l-1-2\gamma_i} \nu_{q^{1/2}}^2 \!\! \sum_{1\leq j,r <n} \!\! \overline{\gamma}_{nj} \gamma_{nr} q^{r-j}e^{+}_{i_1}  \wedge \! \cdots \!  \wedge e^{+}_j \wedge e^{+}_{j'} \wedge \! \cdots \! \wedge e^{+}_{i_l} \wedge e^{-}_{i_1} \wedge \! \cdots \! \wedge e^{-}_r \wedge e^{-}_{r'} \wedge \! \cdots \! \wedge e^{-}_{i_l} \\
&= \! (-1)^l q^{l-1-2\gamma_i} \nu_{q^{1/2}}^2 \!\! \sum_{1\leq j <n } \!\! |\gamma_{nj}|^2 e^{+}_{i_1} \wedge \cdots  \wedge e^{+}_j \wedge e^{+}_{j'} \wedge \cdots \wedge e^{+}_{i_l} \wedge e^{-}_{i_1} \wedge \cdots \wedge e^{-}_j \wedge e^{-}_{j'} \wedge \cdots \wedge e^{-}_{i_l}\\
&+ \! (-1)^l q^{l-1-2\gamma_i} \nu_{q^{1/2}}^2 \!\! \sum_{\substack{1\leq j,r <n \\ j\neq r } } \!\! \overline{\gamma}_{nj} \gamma_{nr} q^{r-j}e^{+}_{i_1} \! \wedge \! \cdots  \! \wedge e^{+}_j \wedge e^{+}_{j'} \wedge \! \cdots \! \wedge e^{+}_{i_l} \wedge e^{-}_{i_1} \wedge \! \cdots \! \wedge e^{-}_r \wedge e^{-}_{r'} \wedge \! \cdots \! \wedge e^{-}_{i_l},
\end{align*} 
where $\gamma_n\in \mathbb{Z}$. Analogously, we can move $e^{+}_j$, $e^{-}_j$ to the left and $e^{+}_{j'}$ and $e^{-}_{j'}$ to the right by using \eqref{ii} and \eqref{holii}, respectively.\\
\\
\textbf{Case} $i>i_l$ and $i=i'_k$ for some $k$: then we can move $e^{+}_i$ to the right by using \eqref{ii} and we get
\begin{align*}
&e^{+}_{i_1} \wedge \cdots \wedge e^{+}_{i_l} \wedge e^{+}_{i} \wedge e^{-}_i \wedge e^{-}_{i_1} \wedge \cdots \wedge e^{-}_{i_l} \\
&= (-1)^{k-1}q^{k-1} e^{+}_{i_1} \wedge \cdots \wedge e^{+}_{i_l} \wedge e^{+}_{i} \wedge e^{-}_{i_1} \wedge \cdots \wedge e^{-}_i \wedge e^{-}_{i_k} \wedge \cdots  \wedge e^{-}_{i_l} \\
&= (-1)^{k-1}q^{k-1}\nu_q \sum_{j<{i_k}}\gamma_{i_k j}q^{j-i_{k}+1}  e^{+}_{i_1} \wedge \cdots \wedge e^{+}_{i_l} \wedge e^{+}_{i} \wedge e^{-}_{i_1} \wedge \cdots \wedge  e^{-}_j \wedge e^{-}_{j'} \wedge \cdots  \wedge e^{-}_{i_l} \\
&\quad -(-1)^{k-1}q^{k-1} e^{+}_{i_1} \wedge \cdots \wedge e^{+}_{i_l} \wedge e^{+}_{i} \wedge e^{-}_{i_1} \wedge \cdots \wedge e^{-}_{i_k} \wedge e^{-}_{i} \wedge \cdots  \wedge e^{-}_{i_l} \\
&= (-1)^{k-1}q^{k-1}\nu_q \sum_{j_s<{i_k}}\gamma_{i_k j_s}q^{j_{s}-i_{k}+1+\theta_{is}} e^{+}_{i_1} \! \wedge \! \cdots \! \wedge e^{+}_{i_l} \wedge e^{+}_{i} \wedge e^{-}_{i_1} \wedge \! \cdots \! \wedge  e^{-}_{j_s} \wedge \! \cdots \! \wedge e^{-}_{j'_s} \wedge \! \cdots \! \wedge e^{-}_{i_l} \\
&\quad +(-1)^{l}q^{l-1} e^{+}_{i_1} \wedge \cdots \wedge e^{+}_{i_l} \wedge e^{+}_{i} \wedge e^{-}_{i_1} \wedge \cdots \wedge e^{-}_{i_k}  \wedge \cdots  \wedge e^{-}_{i_l} \wedge e^{-}_i.
\end{align*}
where $\theta_{is} \in \mathbb{Z}$.\\
\\
\textbf{Case} $I\neq J$: in the same way it can be show that $e^{+}_I \wedge e^{+}_i \wedge e^{-}_i \wedge e^{-}_J$ is a linear combination of elements of the form $e^{+}_P \wedge e^{-}_P$ and $e^{+}_S \wedge e^{-}_T$ for $S\neq T$ where $P, S, T\in \Theta(l+1)$. \\

Finally, if $I'\in \Theta(l+1)$ then from the previous cases we see that the coefficient $f_{l+1,I',I'}(q)$ of $e^{+}_{I'} \wedge e^{-}_{I'}$ is of the form
\begin{align*}
(-1)^{l}\sum_{t}f_{l,I_t,I_t}(q)q^{t} + (q-q^{-1})\xi(q) + (q^{1/2}-q^{-1/2})^2 \zeta(q) + \sum_{r}\chi_r(q) f_{l,I_r,J_r}(q),
\end{align*}
where $I_t, I_r, J_r \in  \Theta(l)$ and $\xi(q), \zeta(q), \chi_r(q) \in \mathbb{C}[q,q^{-1}]$. Since all the $f_{l,I_t, I_t}(1)$ have the same sign and $f_{l,I_r,J_r}(1)=0$, then $f_{l+1,I',I'}(1)=(-1)^{l}\sum_t f_{l, I_t, I_t}(1)\neq 0$ and 
\begin{align*}
\mathrm{sgn}(f_{l+1,I',I'}(1))=(-1)^{l}\mathrm{sgn}\big( \sum_t f_{l, I_t, I_t}(1) \big) =(-1)^{l}(-1)^{\frac{l(l-1)}{2}}=(-1)^{\frac{l(l+1)}{2}}. 
\end{align*}
\end{proof}


\section{Spectrum of the $\adel$-Laplace operator on zero forms}

In this last section we present the proof of the main result of the paper by using previous results and the representation theory of the quantum group $U_q(\mathfrak{so}_N)$.

\begin{lem}\label{Consts}
There exist constants $\theta$, $\theta'$, $\theta''$ and $\theta'''$ such that
\begin{enumerate}[(i)]
\item $-\ast_{\kappa}(\partial y \wedge \ast_{\kappa}(\overline{\partial}y))=\theta y^2 $, \label{adely2}
\item $-\ast_{\kappa}(\ast_{\kappa}( \overline{\partial}y ) \wedge \partial z)=\theta' yz$,
\item $-\ast_{\kappa}(\partial y \wedge \ast_{\kappa}(\overline{\partial}z))= \theta'' yz$,
\item $-\ast_{\kappa}(\ast_{\kappa}(\overline{\partial}z)\wedge \partial z)= \theta'''z^2$, \label{adelz2}
\end{enumerate}
Moreover, these constants 
are non-zero.
\end{lem}
\begin{proof}
We first note that both sides of the equations \eqref{adely2}-\eqref{adelz2} are highest weight vectors with the same weight. By Lemma \ref{Sphewh} and appendix \ref{tablesph} the decomposition  
\begin{align*}
\Omega^{(0,0)}= \cO_q(\mathbf{Q}_N) = \bigoplus_{k,l \in \mathbb{N}_0}U_q(\mathfrak{so}_N) \btr (z^k y^l) 
\end{align*}
is multiplicity-free and then there are constants $\theta, \theta', \theta'', \theta'''$ such that equations \eqref{adely2}-\eqref{adelz2} hold.

To show that $\theta$ is non-zero, we note first that $\del y \wedge \ast_{\kappa}(\adel y) \neq 0$ if and only if $\del y \wedge \adel y $ is a non-primitive form. This follows from the fact that
\begin{align*}
\del y \wedge \ast_{\kappa}(\adel y) = -i^{-1}\frac{1}{(n-1)!} \del y \wedge L^{n-1}(\adel y) = -i^{-1}\frac{1}{(n-1)!}L^{n-2+1}(\del y \wedge \adel y).
\end{align*}
A similar argument shows that $\ast_{\kappa}(\overline{\partial}y ) \wedge \partial z$, $\partial y \wedge \ast_{\kappa}(\overline{\partial}z)$ and $\ast_{\kappa}(\overline{\partial}z)\wedge \partial z$ are non-zero if and only if $\overline{\partial}y  \wedge \partial z$, $\partial y \wedge \overline{\partial}z$ and $\overline{\partial} z \wedge \partial z $ are non-primitive, respectively.

Since the Lefschetz map is a right $U_q(\mathfrak{so}_N)$-module map, if $\omega \in \Om^{(1,1)}$ is primitive then $X \btr \, \omega$ is primitive for any $X\in U_q(\mathfrak{so}_N)$.  
We claim that there exists $X\in U_q(\mathfrak{so}_N)$ such that 
\begin{align*}
[X \btr (\del y \wedge \adel y )] = \gamma [\del z_{jN}] \wedge [\adel z_{Nj}] ~~~~~~~~ \text{for some} \ \  j \in \{2, \ldots, N-1\} \ \ \text{and} \ \ \gamma\neq 0.
\end{align*}
or equivalently there exists $X' \in U_q(\mathfrak{so}_N)$ and $\gamma'\neq 0$ such that 
\begin{align*}
[(\del y \wedge \adel y) \ltr X' ] = \gamma' [\del z_{iN}] \wedge [\adel z_{Ni}] .
\end{align*}
We show this for the case $\mathfrak{so}_{2n+1}$ since for $\mathfrak{so}_{2n}$ the proof is similar.
We define the sequence
\begin{align*}
& X_1 := F_2, \, X_2 := F_3, \ldots , X_{n-1}:=F_n, \, X_n:=F_n, \, X_{n+1}:= F_{n-1}, \ldots, X_{2n-2}:=F_2, \\
& ~ X_{2n-1}:=F_2,\ldots, X_{3n-3}:=F_n, \, X_{3n-2}:=F_n, \, X_{3n-1}:= F_{n-1}, \ldots, X_{4n-4}:= F_2, \\
& ~~~~~~~~~~~~~~~~~~~~~~~~~~~~~~~~~ X_{4n-3}:=F_1, \, X_{4n-2}:=F_1.
\end{align*}
If $1\leq i_{1} < \cdots < i_r \leq 4n-2$ then by using the formulas \eqref{actEioddr}-\eqref{actEnoddr} it can be shown that if $y\ltr X_{i_1} \cdots X_{i_r} $ is non-zero then it is a multiple of one element of the following list
\begin{enumerate}[{(i')}]
\item  $z_{i,N} - q^{\gamma_i} z_{1,i'}$ \ \ $2\leq i \leq N-1$, \ \ $\gamma_i \in \mathbb{Z}$, \label{i'}
\item $z_{N,N}-q^{\eta_1}z_{N-1, N-1} +q^{\eta_2}z_{2,2} - q^{\eta_3}z_{1,1} $, \ \ $\eta_i \in \mathbb{Z}$,
\item $z_{i, N-1} -q^{a_{i}}z_{2,i'}$, \ \ $2\leq i \leq N-2$, \ \ $a_i \in \mathbb{Z}$,
\item $z_{N,N-1}- \mu z_{2,1}$, \ \ $\mu <0$, \label{iv'}
\end{enumerate}
and the only way to obtain $z_{N,N-1} - \mu z_{2,1}$ is by acting by $F_2\cdots F^2_n \cdots F_2F^2_1$ on $y = q(z_{2,N}-q^{-2}z_{1,N-1})$ from the right. Now we note that for any pair of weight elements $\omega, \nu \in \Omega^{\bullet}$ and $W = W_{1} \cdots W_{m} \in U_q(\mathfrak{g})$ we have that 
\begin{align*}
(\omega \wedge \nu ) \ltr W = \sum_{r} \sum_{p\in \mathcal{P}_{m,r}}  q^{\alpha_r} (\omega \ltr W_{p(1)} \cdots W_{p(r)}) \wedge (\nu \ltr W_{p(r+1)} \cdots W_{p(m)})  
\end{align*}
where $q^{\alpha_r}$ are constants obtained by acting with the Cartan elements $K_i$ on $\omega, \nu$ and $\mathcal{P}_{m,r}$ is the set of $(r,n-r)$-shuffles, that is, the set of permutations $p\in \mathcal{P}_{m}$ such that $p(1)< \cdots < p(r)$ and $p(r+1)< \cdots < p(m)$.

Now if we set $X:=X_1 \cdots X_{4n-2}$ then by using the previous list (\ref{i'}')-(\ref{iv'}') and Lemma \ref{Sphewh} (ii), we obtain 
\begin{align}\label{eNeN}
\nonumber [(\del(y)  \wedge \adel(y) ) \ltr X ]&\! = 
\sum_{\substack{r=0 \\ p\in \mathcal{P}_{4n-2,r}}}^{4n-2} q^{\alpha_r} [\del(y \ltr X_{p(1)} \cdots X_{p(r)})] \wedge [\adel(y \ltr X_{p(r+1)} \cdots X_{p(4n-2)})] \\
\nonumber &= q^{\alpha} [\del( y \ltr F_2 \cdots F_n^2 \cdots F_2)] \wedge [\adel (y \ltr F_2 \cdots F_n^2 \cdots F_2 F_1^2)] \\
\nonumber &= \gamma [\del(z_{N-1,N} \!-\!q^{\gamma_{N-1}}z_{1,2}) ] \wedge [\adel ( z_{N,N-1} \!-\! \mu z_{2,1})] \\
&=\gamma e^+_{N-2} \wedge e^{-}_{N-2}, 
\end{align}
where $\gamma <0$ and for the second identity we used the fact that many terms of the sum vanish when we take the quotient $[ \ \cdot \ ]$, and 
we also used the fact the $[\del z_{1,2}]=[\adel z_{2,1}]=0$ in the last line. 
Since the calculus has dimension $2(N-2)$, we claim that the form $e^{+}_{N-2} \wedge e^{-}_{N-2}$ is non-primitive, indeed by Lemma \ref{Kahlerform} we have
\begin{align*}
[\kappa^{N-3} ] \wedge e^{+}_{N-2} \wedge e^{-}_{N-2} & = e^{+}_{N-2} \wedge [\kappa]^{N-3} \wedge e^{-}_{N-2} \\
&= \sum_{I,J \in \Theta(N-3)}f_{N-3,I,J}(q) e^{+}_{N-2} \wedge e^{+}_I \wedge e^{-}_J \wedge e^{-}_{N-2}  \\
&= \sum_{I \in \Theta(N-3)}f_{N-3,I,J}(q) e^{+}_{N-2} \wedge e^{+}_I \wedge e^{-}_I \wedge e^{-}_{N-2} \\
& \quad + \sum_{I \neq J \in \Theta(N-3)}f_{N-3,I,J}(q) e^{+}_{N-2} \wedge e^{+}_I \wedge e^{-}_J \wedge e^{-}_{N-2}\\
&= f_{N-3, J, J}(q) e^{+}_{N-2}\wedge e^{+}_1 \wedge \cdots \wedge e^{+}_{N-3} \wedge e^{-}_{1} \wedge \cdots \wedge e^{-}_{N-3} \wedge e^{-}_{N-2}  \\
&= (-1)^{N-3}q^{-(N-3)} f_{N-3, J, J}(q)  e^{+}_1 \wedge \cdots \wedge e^{+}_{N-2} \wedge e^{-}_{1} \wedge \cdots \wedge e^{-}_{N-2}, 
\end{align*}
where $J=\{1,\ldots, N-3 \}$ and $f_{N-3, J, J}(q)\neq 0$ if $q$ belongs to a suitable interval around $1$. Since $e^{+}_{N-2} \wedge e^{-}_{N-2}$ is non-primitive then $(\del(y)  \wedge \adel(y) ) \ltr X_1 \cdots X_{4n-2}$ is non-primitive and therefore $\del(y)  \wedge \adel(y)$ is non-primitive. The calculation for $\adel y \wedge \del z$ is similar by using the fact that 
\begin{align*}
(\adel y \wedge \del z )\ltr F_1 =q \adel(z_{2,N}-q^{-2}z_{1,N-1}) \wedge \del( z_{2,N}+qz_{1,N-1}), 
\end{align*}
and we can proceed by acting with $X$ and obtain
\begin{align*}
[(\adel y \wedge \del z ) \ltr F_1 X] &=[(\adel y \wedge \del z )\ltr F_1(F_2 \cdots F_n^2 \cdots F_2)^2F_1^2 ] \\
&= q^{\beta} [\adel (y \ltr F_2 \cdots F_n^2 \cdots F_2 F_1^2)] \wedge [\del \big( (z_{2,N}+qz_{1,N-1}) \ltr F_2 \cdots F_n^2 \cdots F_2 \big) ] \\
&= \gamma' [\adel(z_{N,N-1} -\mu z_{2,1})] \wedge [ \del (z_{N-1,N} + q^{\gamma'_{N-1}}z_{1,2})] \\
&= \gamma' e^{-}_{N-2} \wedge e^{+}_{N-2},
\end{align*} 
where $\gamma'\neq 0$. Same calculation works for $\del y \wedge \adel z$
\begin{align*}
[(\del y \wedge \adel z ) \ltr F_1 X] &= [(\del y \wedge \adel z )\ltr F_1(F_2 \cdots F_n^2 \cdots F_2)^2F_1^2 ] \\
&= q^{\beta} [\del (y \ltr F_2 \cdots F_n^2 \cdots F_2 )] \wedge [\adel \big( (z_{2,N}+qz_{1,N-1}) \ltr F_2 \cdots F_n^2 \cdots F_2 F_1^2  \big) ] \\
&= \gamma'' [\del(z_{N-1,N} - q^{\gamma_{N-1}} z_{2,1})] \wedge [ \adel (z_{N,N-1} + q^{\gamma''_{N-1}}z_{1,2})] \\
&= \gamma'' e^{+}_{N-2} \wedge e^{-}_{N-2},
\end{align*} 
for some $\gamma''\neq 0$. From this we can conclude that $\theta'$ and $\theta''$ are not zero.

The fact that $e^{-}_{N-2} \wedge e^{+}_{N-2}$ is non-primitive is shown by using the fact that the K\"ahler form $\kappa$ is non-zero multiple of
\begin{align*}
\mathrm{i} \sum_{i,j=1}^N q^{2(N-i)} \adel z_{ij} \wedge \del z_{ji}
\end{align*}
satisfying analogous results as those of Lemma \ref{Kahlerform}, i.e., there exist $g_{l,I,J}(q)\in \mathbb{C}[q,q^{-1}]$ such that
\begin{align*}
[\kappa^{l}]= \mathrm{i}^{l} \sum_{I,J \in \Theta(l)} g_{l,I,J}(q) e^{-}_I \wedge e^{+}_J, 
\end{align*}
where $\mathrm{sgn}(g_{l,I,I}(1)) = (-1)^{\frac{l(l-1)}{2}}$ for all $I\in \Theta(l)$
and $g_{l,I,J}(1)=0$, $I\neq J$.

Much in the same way is the calculation for $\adel(z) \wedge \del z$: if we act with $F_1$ from the right we obtain 
\begin{align}\label{zzF1}
\nonumber (\adel z_{1,N} \wedge \del z_{1,N}) \ltr F_1^2& = \alpha \adel(z_{2,N} + q z_{1,N-1}) \wedge \del( z_{2,N}+ q z_{1,N-1}) + \beta  \adel( z_{1,N}) \wedge \del(z_{2,N-1}) \\
&\quad + \zeta  \adel(z_{2,N-1})\wedge \del(z_{1,N}),
\end{align}
where $\alpha, \beta, \zeta >0$. Now a routine calculation shows that
\begin{align}\label{z2NX}
& \big[\big(\adel( z_{1,N}) \wedge \del(z_{2,N-1}) \big) \ltr (F_2 \cdots F_n^2 \cdots F_2  )^2F_1^2 \big] =0, \\
\nonumber & \big( \adel ( z_{2,N-1}) \wedge  \del(z_{1,N}) \big) \ltr  (F_2 \cdots  F_n^2 \cdots F_2  )^2F_1^2 = \chi \big( -\adel(z_{N,2}+q^{\alpha} z_{N-1,1} ) \wedge \del(z_{2,N}+q^{\beta}z_{1,N-1}) \\
& \quad + \vartheta_1 \adel(z_{N-1,2}) \wedge \del( z_{2,N-1}) + \vartheta_2 \adel(z_{N,1})\wedge \del(z_{1,N}) \big), 
\end{align}
for a certain real $\chi>0$ and real $\vartheta_1$, $\vartheta_2$. On the other hand, similar calculations as given in \eqref{eNeN} show that
\begin{align}\label{z1+NX}
\big[ \big( \adel(z_{2,N} + q z_{1,N-1}) \wedge \del( z_{2,N}+ q z_{1,N-1}) \big) \ltr (F_2 \cdots F_n^2 \cdots F_2 & )^2F_1^2 \big] = -\xi e^{-}_{N-2}\wedge e^{+}_{N-2},
\end{align}
for a real $\xi>0$. Now combining equations \eqref{zzF1}-\eqref{z1+NX} we have
\begin{align*}
\big[ (\adel z_{1,N} \wedge \del z_{1,N}) \ltr F_1^2(F_2 \cdots F_n^2 \cdots F_2 & )^2F_1^2 \big] = -xi e^{-}_{N-2}\wedge e^{+}_{N-2} - \chi e^{-}_1 \wedge e^{+}_{1}
\end{align*}
which is non-primitive since $-\xi$ and $-\chi$ are both negative and we are assuming that $q$ belongs to a suitable open interval $I$ around $1$.
Finally, similar calculations work when $\mathfrak{g}=\mathfrak{so}_{2n}$ where in this case the element $X$ is given by
$X:= (F_2 \cdots F_n F_{n-2} \cdots F_2)^2 F_1^2.$

\end{proof}

With this lemma and results of previous sections we are now ready to calculate the eigenvalues of the Laplace operator $\Delta_{\adel}$ on zero forms.

\begin{prop}
For $q\in I$ the eigenvalue of the Dolbeault Laplace operator $\Delta_{\adel}$ on zero forms goes to infinity with finite multiplicity and therefore the operator has compact resolvent.
\end{prop}
\begin{proof}
By Lemma \ref{Sphewh},(ii) we have a multiplicity-free decomposition
\begin{align*}
\mathcal{O}_q(\mathbf{Q}_n) = \bigoplus_{k,l\in \mathbb{N}_0} U_q(\mathfrak{so}_N) \btr (z^{k}y^{l}).
\end{align*}
Since $\Delta_{\adel}(z^{k}y^{l})$ is a weight element with the same weight as $z^{k}y^{l}$, multiplicity-freenes implies that $z^k y^l$ is an eigenvector of $\Delta_{\adel}$, in particular $z$ and $y$ are eigenvectors of $\Delta_{\adel}$. Now combining Lemma \ref{adelzzadel}, Lemma \ref{Consts} and the identities \eqref{adjoints} we have
\begin{align*}
\Delta_{\overline{\partial}}(y^k z^l) &= \overline{\partial}^{\dagger}\overline{\partial}(y^k z^l) \\[3pt]
&= \overline{\partial}^{\dagger} \big( (k)_{q^2} y^{k-1}\overline{\partial}y z^l + (l)_{q^{-2}}y^k\overline{\partial}z z^{l-1} \big) \\[3pt]
&= - \ast_{\kappa} \circ \partial \big( (k)_{q^2} y^{k-1} \ast_{\kappa}(\overline{\partial}y) z^l + (l)_{q^{-2}}y^k \ast_{\kappa}(\overline{\partial}z) z^{l-1} \big) \\[3pt]
&= -\ast_{\kappa}\big( (k)_{q^2} y^{k-2} \partial y \wedge \ast_{\kappa}(\overline{\partial}y) z^l + (k)_{q^2} y^{k-1}(\partial \circ \ast_{\kappa} \circ \overline{\partial}y) z^l \\
& \quad + (k)_{q^2}(l)_{q^2} y^{k-1} (\ast_{\kappa} \circ \overline{\partial}y ) \wedge \partial z z^l  + (l)_{q^{-2}}(k)_{q^{-2}}y^{k-1}\partial y \wedge \ast_{\kappa}(\overline{\partial}z)z^{l-1} \\
& \quad +(l)_{q^{-2}}y^{k} \partial \circ \ast_{\kappa}(\overline{\partial}z)z^{l-1} + (l)_{q^{-2}}(l-1)_{q^2}y^k \ast_{\kappa}(\overline{\partial}z)\wedge \partial z z^{l-2} \big) \\[3pt]
&= - (k)_{q^2} y^{k-2} \ast_{\kappa}(\partial y \wedge \ast_{\kappa}(\overline{\partial}y)) z^l - (k)_{q^2} y^{k-1} \ast_{\kappa}(\partial \circ \ast_{\kappa} \circ \overline{\partial}y) z^l \\
& \quad - (k)_{q^2}(l)_{q^2} y^{k-1} \ast_{\kappa}((\ast_{\kappa} \circ \overline{\partial}y ) \wedge \partial z) z^l  - (l)_{q^{-2}}(k)_{q^{-2}}y^{k-1}\ast_{\kappa}(\partial y \wedge \ast_{\kappa}(\overline{\partial}z))z^{l-1} \\
& \quad -(l)_{q^{-2}}y^{k} \ast_{\kappa}(\partial \circ \ast_{\kappa}(\overline{\partial}z))z^{l-1}  -(l)_{q^{-2}}(l-1)_{q^2}y^k \ast_{\kappa}(\ast_{\kappa}(\overline{\partial}z)\wedge \partial z) z^{l-2} \\[3pt]
& = \big[ \theta(k)_{q^2}(k-1)_{q^{-2}} + (k)_{q^2}\mu_{y} + (l)_{q^2}(k)_{q^2}\theta' + (l)_{q^{-2}}(k)_{q^{-2}}\theta'' + (l)_{q^{-2}}\mu_z \\
& \quad + (l)_{q^{-2}}(l-1)_{q^2}\theta''' \big] y^k z^l,
\end{align*}
where $\mu_z$ and $\mu_y$ are the eigenvalues of $z$ and $y$, respectively and $(n)_q$ is given by
\begin{align*}
(n)_q:= \frac{1-q^n}{1-q} = q^{(n-1)/2}[n]_{q^{1/2}},
\end{align*}
being another version of $q$-numbers. A differential calculus $(\Omega^{(\bullet,\bullet)},\del,\adel)$ is said to be \emph{connected} if 
\begin{align*}
\text{Ker}(\del: \Om^{(0,0)} \rightarrow \Om^{(1,0)}) = \text{Ker}(\adel: \Om^{(0,0)} \rightarrow \Om^{(0,1)}) = \mathbb{C}1. 
\end{align*}
Since for any irreducible quantum flag manifold the Heckengerger--Kolb calculus is connected \cite[Theorem 4.4]{AFO} and $\mathrm{ker}(\Delta_{\adel})= \mathrm{ker}(\adel^{\dagger})\cap \mathrm{ker}(\adel)$ \cite[Lemma 6.1]{RO}, then $\mu_z$ and $\mu_y$ are positive. First we assume that $q>1$, then $(l)_{q^{-2}}(k)_{q^{-2}}\theta''$ and  $(l)_{q^{-2}}\mu_z $ have finite limit as $k,l\rightarrow \infty$ and therefore they do not contribute in the asymptotic behaviour of the eigenvalue. 
By Lemma \ref{Consts} we have that $\theta'\neq 0$. It is not difficult to see that if $\theta' < 0$, $k$ is fixed and sufficiently large and $l\rightarrow \infty$ then the eigenvalue would tend to $-\infty$ and this would contradict the fact that the eigenvalues of the Laplacian are non-negative. Therefore $\theta' >0$ and the term $(l)_{q^2}(k)_{q^2}\theta' \rightarrow \infty$ as $k,l \rightarrow \infty$. Now we consider the term
\begin{align}\label{theta'=0}
\big( \theta \mu_y^{-1}(k -1)_{q^{-2}}+1 \big)(k)_{q^2}\mu_y. 
\end{align}
First we note that $\theta < -(1-q^{-2})\mu_y$ implies that
\begin{align*}
\big(1+(k-1)_{q^{-2}} \theta \mu_y^{-1}\big)<0, ~~~~~ \text{as} \ \ k\rightarrow \infty,
\end{align*}
and this would imply that the eigenvalues of $\Delta_{\adel}$ on $y^k$ are negative. Therefore $\theta \geq -(1-q^{-2})\mu_y$. If $\theta = -(1-q^{-2})\mu_y$ then $(\theta \mu_y^{-1}(k -1)_{q^{-2}}+1 )(k)_{q^2}\mu_y$ converges when $k \rightarrow \infty$. On the other hand, if $\theta > -(1-q^{-2})\mu_y$ then $(\theta \mu_y^{-1}(k -1)_{q^{-2}}+1 )(k)_{q^2}\mu_y$ tends to infinity. This means that in any case the value of $\theta$ does not affect the required asymptotic behaviour of the eigenvalue. 

Finally if $\theta''' <0$ then 
\begin{align*}
\Delta_{\overline{\partial}}(z^l)= \theta'''(l)_{q^{-2}}(l-1)_{q^2} + (l)_{q^{-2}}\mu_z = (\mu_z^{-1}\theta'''(l-1)_{q^2} +1)\mu_z (l)_{q^{-2}} < 0,
\end{align*} 
as $l \rightarrow \infty$ which contradicts the fact the eigenvalues of $z^l$ are non-negatives. Then $\theta'''\geq 0$. By Lemma \ref{Consts} we have that $\theta''' >0$ and this implies that the term $(l)_{q^{-2}}(l-1)_{q^2}\theta'''$ 
tends to infinity as $l \rightarrow \infty$. Combining the above results we can conclude that the eigenvalues of $\Delta_{\adel}$ tend to infinity, as $k,l \rightarrow \infty$.
For the case $q<1$, by using a similar argument as that of $\theta'$ it can be shown that $\theta''$ is positive. Since the term $(l)_{q^{-2}}(k)_{q^{-2}}\theta''$ grows faster than the others, then the eigenvalue goes to infinity when $k,l \rightarrow \infty$.   
\end{proof}


\appendix

\section{Spherical weights} \label{tablesph}

With respect to the left action defined in Lemma \ref{Sphewh}, the weights of the highest weight vectors of $\cO_q(G/L_S)$ form an additive monoid $Z_S \subset \mathfrak{h}^{*}$. By \cite[Proposition 4.1]{DST} the minimal number of generators on $Z_S$, considered as monoid, is the same as the classical case $q=1$. This minimal set of generators are called the \emph{spherical weights} which was presented by Kr\"amer \cite[Tabelle 1]{MK} in the classical case and we present here in Table \ref{sphericalweights} below for each irreducible quantum flag manifold.

We remark that for the case of $\cO_q(\mathbf{S}_{2n})$, the weight $2\varpi_{2n-1}$ or $2\varpi_{2n}$ appears according to the defining crossed node, as presented in table \ref{table:CQFMs} above. Finally we note that in part (ii) of the Lemma \ref{Sphewh} we used the fact that the weights $2\varpi_1, \varpi_2$ form a minimal set of generators for  the weights of the highest weights vectors of $\cO_q(\mathbf{Q}_N)$. 

\begin{center}
\begin{table}[ht]

\centering

\captionof{table}{\small{Spherical weights according to the numbering of the Dynkin nodes in \cite[\textsection 11.4]{JH}.}}\label{sphericalweights}
{\small \renewcommand{\arraystretch}{2}%
\begin{tabular}{|c|c|}
 
\hline

\small $\cO_q(\text{Gr}_{r,s})$ & $\varpi_1 + \varpi_{r+s-1}, \varpi_2 + \varpi_{r+s-2}, \ldots, \varpi_r + \varpi_{s}$ \\


\small $\cO_q(\mathbf{Q}_{2n+1})$ & $2\varpi_1, \varpi_2$  \\ 


\small $\cO_q(\mathbf{L}_{n})$ & $2\varpi_1, 2\varpi_2, \ldots, 2\varpi_n$   \\ 


\small $\cO_q(\mathbf{Q}_{2n})$ & $2\varpi_1, \varpi_2$  \\ 


\small $\cO_q(\textbf{S}_{2n})$ & $\varpi_2, \varpi_4, \ldots, \varpi_{2n-2}, 2\varpi_{2n-1}$ or $2\varpi_{2n}$  \\


\small $\cO_q(\textbf{S}_{2n+1})$ & $\varpi_2, \varpi_4, \ldots, \varpi_{2n-2}, \varpi_{2n} + 2\varpi_{2n+1}$  \\


\small  $\cO_q(\mathbb{OP}^2)$ & $\varpi_1 +\varpi_6, \varpi_2$  \\


\small   $\cO_q(\textrm{F})$ & $2\varpi_7, \varpi_1, \varpi_6$  \\
\hline
\end{tabular}
}
\end{table}
\end{center}

\subsection*{Acknowledgments} 
FDG was supported by the Charles University PRIMUS grant \emph{Spectral Noncommutative Geometry of Quantum Flag Manifolds} PRIMUS/21/SCI/026. The author would like to thank R\'eamonn \'O Buachalla for useful discussions during the preparation of this paper.



\end{document}